\documentclass[11pt]{amsart}
\usepackage{amssymb, amsthm, amscd, epsfig, graphicx, 
enumerate, stmaryrd, amsmath, graphs, mathtools}

\newtheorem{thm}{Theorem}[section]
\newtheorem{lem}[thm]{Lemma}
\newtheorem{prop}[thm]{Proposition}

\theoremstyle{definition}
\newtheorem{defn}[thm]{Definition}

\newtheorem{cor}[thm]{Corollary}
\theoremstyle{remark}
\newtheorem{rem}[thm]{Remark}

\numberwithin{equation}{section}

\newcommand{\al}{\alpha}
\newcommand{\be}{\beta}
\newcommand{\ga}{\gamma}
\newcommand{\Ga}{\Gamma}

\newcommand{\ep}{\varepsilon}

\newcommand{\ka}{\kappa}

\newcommand{\la}{\lambda}

\newcommand{\si}{\sigma}
\newcommand{\Si}{\Sigma}
\renewcommand{\th}{\theta}

\newcommand{\va}{\varphi}
\newcommand{\csi}{\xi}

\newcommand{\x}{\times}

\newcommand{\CC}{\mathcal C}
\newcommand{\EE}{\mathcal E}
\newcommand{\imm}{{\mathrm {Imm}}}

\newcommand{\Z}{\mathbb Z}
\newcommand{\N}{\mathbb N}

\newcommand{\R}{\mathbb R}

\newcommand{\CP}{{\mathbb C}{P}}

\newcommand{\del}{\partial}

\newcommand{\lra}{\longrightarrow}
\newcommand{\co}{\colon\thinspace}

\newcommand{\im}[1]
 {{\mathrm {im}}\thinspace{#1}}

\renewcommand{\min}[2]
 {{\mathrm {min}}({#1},{#2})}

\renewcommand{\int}{{\mathrm {int}}}

\begin{document}
\mathsurround=1pt 
\title%[Fold  maps and  the signature]
%{A Poincar\'e-Hopf type theorem for the signature and cobordism of fold maps}
{Fold cobordisms %stably framed manifolds, 
and 
a Poincar\'e-Hopf type theorem for the signature}



\thanks{The author was supported 
%by the JSPS Research Fellowship for Young Scientists,
%then 
by OTKA NK81203 and by the Lend\"ulet program 
of the Hungarian Academy of Sciences.}

\author{Boldizs\'ar Kalm\'{a}r}

\address{Alfr\'ed R\'enyi Institute of Mathematics,
Hungarian Academy of Sciences \newline
Re\'altanoda u. 13-15, 1053 Budapest, Hungary}
\email{boldizsar.kalmar@gmail.com}

%
%\address{Budapest University of Technology and Economics \newline
%M\H{u}egyetem rkp. 3, 1111 Budapest, Hungary}

%\address{E\"{o}tv\"{o}s Lor\'{a}nd University, P\'{a}zm\'{a}ny P\'{e}ter
%s\'{e}t\'{a}ny 1/c.  H-1117 Budapest, Hungary}
%\email{kalmbold@cs.elte.hu}

\begin{abstract}
We give complete geometric invariants of cobordisms of framed fold maps.
These invariants are of two types. We take the immersion of the fold singular set into the target manifold
together with information about non-triviality of the normal bundle of the singular set in the source manifold.
These invariants were introduced in the author's earlier works.
Secondly we take the induced stable partial framing on the source manifold whose cobordisms were studied in general
by Koschorke. We show that these invariants describe completely the cobordism groups of framed fold maps into $\R^n$.
Then we are looking for dependencies between  these invariants
and  we study fold maps of $4k$-dimensional manifolds into $\R^2$.
We construct special fold maps which are
representatives of the fold cobordism classes
and we also compute the cobordism groups.
We obtain a Poincar\'e-Hopf type formula, which connects
local data of the singularities
of a fold map
of an oriented
$4k$-dimensional manifold $M$ to
the signature of $M$.
We also study the unoriented case analogously and prove a similar formula 
about the twisting of the normal bundle of the fold singular set.
\end{abstract}

\maketitle

%{\SMALL{ \tableofcontents}}

\section{Introduction}\label{introduction}

Let $n \geq 1$ and $q \geq 0$.
%Fold and cusp maps are the simplest smooth maps with singularities. 
A smooth map 
$f$ of an $(n+q)$-dimensional manifold $M$ into $\R^n$ is called 
a {\it fold map} if it
can be written as
$$%\begin{equation}\label{foldnormalform}
f(x_1,\ldots,x_{n+q})=\left(x_1,\ldots,x_{n-1},  \sum_{i=n}^{n+q}{\pm x_i^2}\right)
$$%\end{equation} 
in local charts around each critical point $p \in M$ and critical value $f(p)$. Fold maps are natural generalizations 
of Morse functions, and play a basic role in the theory of singular maps.
For example, it is always possible to deform any map of an at least $2$ dimensional closed orientable  manifold with even  Euler
characteristic  into $\R^2$  to obtain a fold map, see \cite{Eli1, Ka00, Lev1, Mi84}.
%
%An example for the application of cusp maps is in low dimension:
%We call a  map of an oriented $4$-manifold into a surface a {\it generalized
%broken Lefschetz fibration} if it is a cusp map in general position outside of a finite number of (anti)Lefschetz singular points, i.e.,
%singular points which have the local form $(u,v) \mapsto u^2+v^2$ or
%$(u,v) \mapsto {\bar u}^2+v^2$ in complex coordinates.
%In this paper, besides other results we obtain the following.
%\begin{prop}\label{intro4dimben}\noindent
%%\begin{enumerate}
%%\item
%%An isomorphism $\CC ob_{{}}^O(4,-2) \to \Z \oplus \Z_2^2$ is given by the map
%%$$[f \co X^{4} \to \R^2] \mapsto \si(X^{4})/2 \oplus 
%%[f|_{S_0(f)}] \oplus [f|_{S_1(f)}] .$$
%%\item
%If $f \co X^4 \to S^2$ is a generalized broken Lefschetz 
%fibration over the $2$-sphere,
%then $\si( X^{4} )  \equiv  c(f) - l(f) + 2(t(f) + \del([f]))  \mod{4}$
%holds.%, where $l(f)$ denotes the number of (anti)Lefschetz singular points.
%%\end{enumerate}
%\end{prop}

%Here $\si$ denotes the signature, $c(f)$ denotes the algebraic sum of cusps,...

%...egyeb jelolesek...

Given a fold map  $f \co M^{2+q} \to \R^2$, the {\it singular set of $f$}, i.e.\ the set of points in $M$ where the rank of the differential of $f$ 
is less than $2$, is a $1$-dimensional submanifold of $M$. 
A fold map, which we always presume being in generic position, restricted to its singular set is a generic immersion.
%We will always assume that a fold map is in generic position.
In the case of $q$ odd, Chess \cite{Chess} found a relation between the number of double points of this immersion, the twisting of the normal bundle of the singular set in $M$ and certain 
Stiefel-Whitney classes.
Namely, let
%\begin{thm}[\cite{Chess}]
%Let 
$k \geq 1$ and $f \co M^{2k+1} \to \R^2$ be a fold map of a closed orientable manifold.
Then
\[
 t(f) + \tau(f) \equiv \left \{
 \begin{array}{ll}
 0 \mod{2} & \mbox{if $k$ is odd,} \\
 w_2 w_{2k-1}[M^{2k+1}]  \mod{2} & \mbox{if $k$ is even.}
 \end{array}
 \right. \]
%\end{thm}
Here $w_i$ denotes the $i$th Stiefel-Whitney class,
$t(f)$ denotes the number of double points of 
the immersed singular set of $f$ in $\R^2$
%the set of critical values of $f$ in $\R^2$
and $\tau(f)$ measures how non-trivial the normal bundle of the singular set of $f$ in $M$ is, for the precise definitions, see Section~\ref{geomcobinv} of the present paper.

In \cite{Chess} this result is called a Poincar\'e-Hopf type theorem because it relates 
some topological property of a $(4k+1)$-dimensional manifold $M$ to the local behavior of its map into $\R^2$ (which of course corresponds to some local behavior of some vector fields on $M$).

In the present paper, we look for this type of results in the case of $4k$-dimensional manifolds.
We obtain

\begin{thm}\label{introsign}
Let $k \geq 1$ and $f \co M^{4k} \to \R^2$
be a fold map of a closed oriented $4k$-dimensional manifold. Then
$$t(f) + \tau(f)  \equiv   \frac{\si( M^{4k} ) }{2}  \mod{2}.$$
%This congruence also holds for a stable map
% $f \co M^{4k} \to S^2$  of
% a closed oriented $4k$-dimensional manifold into the $2$-sphere if there exists a regular value
% $y \in S^2$ such that the fiber $f^{-1}(y)$ is an oriented null-cobordant $(4k-2)$-dimensional manifold.
 \end{thm}
Here $\si(M^{4k})$ denotes the signature of the closed  oriented $4k$-dimensional manifold $M^{4k}$.
The invariants $t$ and $\tau$ are again the number of double points of the immersed singular set and the twisting of the normal bundle of the singular set of a fold map.

%Note that 
By \cite{Lev1} the manifold $M^{4k}$ has a fold map into $\R^2$ if and only if 
the Euler characteristic of $M^{4k}$ is even. A corollary of Theorem~\ref{introsign} gives the condition for fold maps when 
the signature of $M^{4k}$  is divisible by $4$. 
\begin{cor}
Let $k \geq 1$ and $M^{4k}$ be a closed oriented $4k$-dimensional manifold.
Then the following are equivalent.
\begin{enumerate}[\rm (i)]
\item
The signature $\si(M^{4k})$ is divisible by $4$.
\item
There is a fold map $f \co M^{4k} \to \R^2$ such that $t(f) \equiv   \tau(f)    \mod{2}$.
\end{enumerate}
Each of these implies that
for any fold map $f \co M^{4k} \to \R^2$ we have $t(f) \equiv   \tau(f)    \mod{2}$.
\end{cor}

In the non-orientable case, we have
\begin{thm}\label{introsignunori}
Let $k \geq 1$ and $f \co M^{4k} \to \R^2$
be a fold map of a closed  (possibly non-orientable) $4k$-dimensional manifold. Then
$$\tau^1(f) + \tau^2(f)  \equiv 0  \mod{2}.$$
%This congruence also holds for a stable map
% $f \co M^{4k} \to S^2$  of
% a closed oriented $4k$-dimensional manifold into the $2$-sphere if there exists a regular value
% $y \in S^2$ such that the fiber $f^{-1}(y)$ is an oriented null-cobordant $(4k-2)$-dimensional manifold.
 \end{thm}
Again, for the precise definitions of the invariants $\tau^1$ and $\tau^2$, which measure the twisting of the normal bundle of the singular set of a fold map, see 
Section~\ref{geomcobinv}.
As we mentioned, the relation between the Euler 
characteristic of the source manifold of a map  and the singularities of the map is quite well-known. However, 
the relation between the signatures of  oriented source manifolds and their singular maps is still much less understood.
Existing results in  \cite{ ES03,  Oz02, SY06}  establish  formulas about the signatures of oriented  manifolds and their singular maps
but for only $4$-dimensional source manifolds. In \cite{TY06} formulas about non-oriented $4$-manifolds and their singularities are obtained.

We obtain Theorems~\ref{introsign} and \ref{introsignunori}
by considering cobordisms of fold  maps (see Definition~\ref{cobdef}). 
It is well-known that
two closed manifolds are cobordant if and only if
the corresponding Stiefel-Whitney numbers (and Pontryagin numbers in the oriented case) 
coincide. We have the  notion of a {\it cobordism of singular maps} \cite{Szucs1, RSz}, and 
if there are two given  maps, one can ask about
easily applicable 
procedures, 
namely checking {\it cobordism invariants}, to decide whether the two fold maps are cobordant. 
Fold maps restricted to their singular sets are immersions, and 
we introduced and used geometric cobordism invariants of fold maps in \cite{Kal7, Kal4, Kal6}
 in terms
of these immersions with prescribed normal bundles
which describe the image of the fold singular set in the target manifold 
together with more detailed informations about the ``twisting'' of
the tubular neighborhood of the singular set in the source manifold (we summarize these results in Section~\ref{geomcobinv}).
More precisely, the fold singularities
$$(x_1,\ldots,x_{n+q})  \mapsto \left(x_1,\ldots,x_{n-1}, \sum_{i = n}^{n + \la -1}
-x_i^2  + \sum_{i = n+ \la}^{n+q} x_{i}^2 \right),$$
where $0 \leq \lambda \leq (q+1)/2$,
form an $(n-1)$-parameter family of  the index $\lambda$ Morse singularities $$(x_n,\ldots,x_{n+q})  \mapsto \sum_{i = n}^{n + \la -1}
-x_i^2  + \sum_{i = n+ \la}^{n+q} x_{i}^2.$$ Roughly speaking this Morse singularity is a map $\va \co \R^{q+1} \to \R$ and has a symmetry group
$G_{\lambda}$ acting on $\R^{q+1}$ and $\R$, respectively. The singular locus $S_{\lambda}$ of a fold map $f \co M \to \R^n$ consisting of fold singular points with $\lambda$ many ``$-$'' signs 
is an $(n-1)$-dimensional submanifold of $M$ and the normal bundle of the immersion $f|_{S_{\lambda}} \co S_{\lambda} \to \R^n$ can be induced 
from the line bundle $l^1_{\lambda} \to BG_{\lambda}$ which we get by taking the action of $G_{\lambda}$ on the target $\R$ of the Morse singularity $\va$.
By assigning this immersion to the fold map and by looking at cobordisms, we get a homomorphism
$\csi_{\lambda}$ from the cobordism group of fold maps into the cobordism group $$\imm \left({l^1_{\lambda}},n \right),$$ which denotes the
cobordism group of codimension $1$  immersions into $\R^n$ 
whose normal bundles are induced from $l^1_{\lambda}$.
These homomorphisms $\csi_{\lambda}$ are the geometric cobordism invariants which we introduced in \cite{Kal7, Kal4, Kal6}.

In this paper, we define further invariants 
which describe the cobordism class of the source manifold together with  its mapping { away from}  
 the singular set as well. Namely, for a given  fold map, we take the {\it stably framed cobordism class} of its source manifold (see Section~\ref{ujcob}), which notion
was studied in general by Koschorke \cite{Ko}. More precisely, on the source manifold of a fold map $f \co M^{n+q} \to \R^n$ with oriented singular set
we obtain a stable (partial) framing simply by  pulling back  the parallelization of the 
tangent space $T\R^n$ by the modified differential 
$df + \al \co TM \oplus \ep^1 \to T\R^n$, where $\al \co \ep^1 \to   T\R^n$ is just a homomorphism having full rank  near the singular set of $f$
and consequently $df + \al$ is an epimorphism, see  
Section~\ref{existfrfold} and also \cite{An3, Sa1}.
Looking at the cobordisms of these stable framings, 
we obtain our homomorphism 
 $\si^{}_{n,q}$ which maps the cobordism class of the fold map $f$ to the cobordism class of the stably framed source manifold $M^{n+q}$.

Finally,
by using a result of Ando \cite[Theorem~0.1]{An3} about the existence of %even codimensional 
fold maps, we show that  our invariants give complete cobordism
invariants of 
(framed) fold maps (see Definition~\ref{framedcobdef}, Theorem~\ref{invarithm},
Corollary~\ref{foldcobcor}). 

Namely, for $n>0, q \geq 0$, let us denote 
  by $\Im_{n,q}$ the homomorphism
 $$\left(\si^{}_{n,q} , \csi_{{1}}, \ldots, \csi_{\lfloor (q+1)/2 \rfloor}\right),$$
 which maps the  cobordism class of a (framed) fold map into the direct sum of the 
 cobordism group of stably framed manifolds and the 
 groups $\imm \left({l^1_{\lambda}},n \right)$ for $1 \leq \lambda \leq (q+1)/2$ as described above.
 
 Then we obtain the following result.
\begin{thm}\label{hominj}
The homomorphism
$\Im_{n,q}$
is injective. 
\end{thm}
It is important to note that Theorems~\ref{introsign} and \ref{introsignunori} and also the result of \cite{Chess} show dependencies between 
the cobordism invariants $\csi_0, \ldots, \csi_{\lfloor (q+1)/2 \rfloor}$ and $\si_{n,q}$.

We prove Theorem~\ref{hominj} by applying the h-principle of Ando \cite{An3} in a short and simple way.
We have the analogous results when we consider orientations on the manifolds and on the cobordisms as well.
When we arrive to applications and computations, what simplifies everything is that the line bundles $l^1_{\lambda}$ are in fact trivial bundles in almost all of the cases.
For example, we obtain
\begin{thm}\label{introcobori}
For $k \geq 1$, 
the cobordism group %$\CC ob_{{}}^{O}(2,4k-2)$ 
of fold maps of oriented $4k$-dimensional manifolds into $\R^2$
is isomorphic to $$\Omega_{4k}^{2|\chi} \oplus \Z_2^{4k-2},$$ 
where 
$\Omega_{4k}^{2|\chi}$ is the cobordism group of closed oriented $4k$-dimensional manifolds with even Euler characteristic.
\end{thm}

For the unoriented case, we have
\begin{thm}\label{introcob}
For $k \geq 1$, the  cobordism group %$\CC ob_{{}}^{}(2,4k-2)$ 
of fold maps of  $4k$-dimensional manifolds into $\R^2$
is isomorphic to $${\mathfrak {N}}_{4k}^{2|\chi} \oplus \Z_2  \oplus \Z_2^{6k-3},$$
where 
${\mathfrak {N}}_{4k}^{2|\chi}$ is the  cobordism group of closed unoriented $4k$-dimensional manifolds with even Euler characteristic.
\end{thm}

To prove our Theorems~\ref{introsign} and \ref{introsignunori},
we also construct representatives of the generators of these cobordism groups of fold maps. %  in Theorems~\ref{introcobori} and \ref{introcob}.
In the  case of Theorem~\ref{introcobori}, the direct summand $\Z_2^{4k-2}$ is generated by the classes of such fold maps which are Morse function bundles over 
immersed $1$-dimensional manifolds in $\R^2$. In \cite{Kal7} we introduced these Morse function bundles to detect direct summands of the 
cobordism groups of fold maps into $\R^n$.
We recall these results in a detailed form in Section~\ref{mobuimm} in the special case of $n = 2$. 
To find representatives of generators of the direct summand $\Omega_{4k}^{2|\chi}$, we need two things.
At first, we apply  \cite[Theorem~3]{AK}  to get fiber bundles over $S^2$  which we map then into $\R^2$ by constructing some fold maps. 
The invariants $t$, $\tau$ and the signature of the source manifold are typically zero for these maps.
Secondly, 
we construct a specific fold map  $f_C \co \CP^{2k} \# \CP^{2k} \to \R^2$, which we use to get more non-zero values for $\tau$ and the signature. 

In the case of Theorem~\ref{introcob}, the direct summand $\Z_2  \oplus \Z_2^{6k-3}$ is generated by the classes of Morse function bundles and the map $f_C$ similarly to 
the oriented case.
The direct summand ${\mathfrak {N}}_{4k}^{2|\chi}$ is generated by classes of fold maps of fiber bundles over $S^2$ 
similarly to the oriented case but now we apply results of \cite{Br69} to get these bundles.

So
having  invariants which encode geometric information, computing the cobordism groups, constructing representatives of the generators
of the cobordism groups, and checking the values of the invariants on them lead to formulas about geometric properties --- in this paper we implement this concept
in the case of oriented $4k$-dimensional manifolds ($k \geq 1$) and their fold maps into the plane. 
%(see Theorem~\ref{invarithm}, Proposition~\ref{explsikbathm}, Section~\ref{mobuimm}, and Proposition~\ref{sign}). 
In this way, we obtain 
Theorem~\ref{introsign}. For the unoriented case, we similarly obtain Theorem~\ref{introsignunori}.

Other results about cobordisms of fold and singular maps can be found for example in 
\cite{An, An6, EST07, Ik, IS03, Kal, Kal2, KT12, RSz, Sad09, Sad4, Sa6, SST10, Szucs1, Szucs4, ST19}.

The paper is organized as follows.
In Sections~\ref{preli}, \ref{kob} and \ref{ujcob}, we give and study the basic definitions,
in Sections~\ref{Complete_invariants_of_cobordisms_of_framed_fold_maps} and \ref{foldmapsplanePoinHopf} we state our main results,
in Section~\ref{mobuimm} we recall necessary  results about Morse function bundles and
in Section~\ref{completebiz}, we prove our results.

\subsection*{Acknowledgements}
The author would like to thank Prof.\ Yoshifumi Ando for the interesting and useful conversations
especially about his Theorem~\ref{relexifold}.
The author  also would like to thank the anonymous referee for the helpful comments, which improved the paper.

\subsection*{Notations}
In this paper the symbol $\amalg$ denotes the disjoint union. 
For any number $x$ the symbol $\lfloor x \rfloor$ denotes the greatest
integer $i$ such that $i \leq x$.
%$\ga^1$ denotes the universal line bundle over $\RP^{\infty}$,
For an integer $k \geq 0$ the symbol $\ep^k_X$ denotes the trivial $k$-dimensional vector bundle over the space $X$.
For a pair of spaces $(X, A)$ and a vector bundle 
$\csi$ over $X$ the symbol $\csi|_A$ denotes 
the  bundle induced by the inclusion $A \subset X$.
For a smooth manifold $X$ the symbol $TX$ denotes its tangent space.
The symbol $\pi_n^s(X)$ (resp.\ $\pi_n^s$) denotes the $n$th stable homotopy group of the space $X$ (resp.\ spheres).
The symbol $\imm({\eta^k},n)$ denotes
the cobordism group of codimension $k$ immersions into $\R^n$ 
whose normal bundles are induced from the vector bundle $\eta^k$ (this
group is isomorphic to $\pi_{n}^s(T\eta^k)$, where $T\eta^k$ is the Thom space of $\eta^k$, see \cite{We}).
The symbol $\Omega_{m}$ (resp.\ ${\mathfrak {N}}_{m}$)
denotes
the cobordism group of oriented (resp.\ unoriented) $m$-dimensional manifolds.
If $W$ is a manifold with boundary $X$, then
the tangent space of the submanifold $X$ is a codimension $1$ subbundle of 
the restriction $TW|_X$.
Any tangent vector $v \in TW|_p$, where $p \in X$, such that $v$ and $TX|_p$ generate $TW|_p$ 
is called a normal section of $X$ at $p$. All manifolds are of class $\mathcal C^{\infty}$.

\section{Preliminaries}\label{preli}

\subsection{Maps with fold singularities}
 
Let $n \geq 1$ and $q \geq 0$. Throughout this section 
let $Q$ and $N$ be smooth manifolds of dimensions $n+q$ and $n$ 
respectively. Let $p \in Q$ be a singular point of 
a smooth map $f \co Q\to N$. The smooth map $f$  has a {\it fold 
singularity of index $\la$} at the singular point $p$ if we can write $f$ in some local coordinates around $p$  
and $f(p)$ in the form 
\[  
f(x_1,\ldots,x_{n+q})=\left(x_1,\ldots,x_{n-1}, \sum_{i = n}^{n + \la -1}
-x_i^2  + \sum_{i = n+ \la}^{n+q} x_{i}^2 \right)
\] 
for some $0 \leq \la \leq q+1$ (the {\it index} $\la$ is well-defined if
we consider that $\la$ and $q+1-\la$ represent the same absolute index). 

%The smooth map $f$  has a {\it cusp
%singularity of index $\la$} at the singular point $p$ if we can write $f$ in some local coordinates around $p$  
%and $f(p)$ in the form 
%\begin{equation}\label{cuspform}
%f(x_1,\ldots,x_{n+q})=(x_1,\ldots,x_{n-1},  x_{n-1} x_{n+q} + x_{n+q}^3 + \sum_{i = n}^{n + \la -1}
%-x_i^2  + \sum_{i = n+ \la}^{n+q-1} x_{i}^2 )
%\end{equation}
%for some $0 \leq \la \leq q$ (the {\it absolute index} $\la$ is well-defined if
%we consider that $\la$ and $q-\la$ represent the same absolute index). 

%
%A smooth map $f \co Q^{n+q} \to N^{n}$ is called a {\it fold map} ({\it cusp map}) if $f$ has only 
%fold (resp.\ fold or cusp) singularities. Note that for $n=2$ a map is stable if and only if it is a  cusp map.

A fold singularity is
%The smooth map $f$ 
%  has 
  a {\it definite fold
singularity}  if $\la = 0$ or $\la = q+1$
and it is an {\it indefinite fold singularity of index $\la$}
 if $1 \leq \la \leq q$.

Let $S_{\la}^{}(f) \subset Q$ denote the set of fold singularities of index $\la$ of $f$.
Note that $S_{\la}^{}(f) = S_{q+1-\la}^{}(f)$.
 Let $S_f^{}$ denote the set $\bigcup_{\la} S_{\la}^{}(f)$ of all the fold singularities of $f$.
Then $S_f$ is a smooth ${(n-1)}$-dimensional submanifold of
$Q$.
If $f \co Q \to N$ is a generic fold map, then 
the restriction of  $f$ to $S_f^{}$ is a generic
 codimension one  immersion  into the target manifold $N$. 
Each connected component of the manifold $S_f^{}$  has its own index $\la$ if
we consider that $\la$ and $q+1-\la$  represent the same index.

Since every fold map is in general position after a small perturbation, 
and we study maps under the equivalence relation {\it cobordism}
(see Definition~\ref{cobdef}),
in this paper we can restrict ourselves to studying fold maps which are 
in general position.
Without mentioning we presume that a fold map $f$ is generic.

For a fold map $f \co Q \to N$ and for an index $0 \leq \la \leq \lfloor q/2 \rfloor$
the normal bundle of the codimension one immersion $f |_{S_{\la}^{}(f)} \co S_{\la}^{}(f) \to N$ 
 has a {\it canonical} orientation or framing 
(i.e.\ trivialization) by identifying the
fold singularities of index $\la$ 
with   
the fold germ 
$(x_1,\ldots,x_{n+q}) \mapsto \left(x_1,\ldots,x_{n-1}, \sum_{i = n}^{n + \la -1}
-x_i^2  + \sum_{i = n+ \la}^{n+q} x_{i}^2 \right)$.

\begin{defn}[Framed fold map]\label{cooricuspmap}
We say that a fold map $f \co Q \to N$ is {\it framed} if
\begin{enumerate}[\rm (1)]
\item
the normal bundle of the codimension $1$  immersion $f|_{S_f^{}} \co S_f^{} \to N$ 
%of  $S_f^{1,j}$ 
is oriented 
%(i.e., trivialization of the normal bundle) 
so that
for each index $0 \leq \la  \leq \lfloor q/2 \rfloor$
the orientation of the normal bundle of the immersion 
$f|_{S_{\la}^{}(f)}$
coincides with the {\it canonical} orientation and
\item
for odd $q$
the normal bundle of the immersion $f|_{S_{(q+1)/2}^{}(f)}$ is orientable and it is oriented. 
\end{enumerate}
\end{defn}

\begin{rem}
\begin{enumerate}[\rm (i)]
\item
If we have a framed fold map $f$ into an oriented manifold $N$, then since
the fold singular set $S_f^{}$ 
is immersed into $N$, the manifold $S_f$ 
has an induced orientation given by this framing and the orientation of the target
manifold $N$. 
\item
A fold map $f \co Q \to N$ can be  framed if and only if 
the cokernel bundle of the differential $df \co TQ \to f^*TN$ restricted to $TQ|_{S_f}$ is a trivial line bundle.
\end{enumerate}
\end{rem}

\begin{defn}[Oriented fold map]\label{orifoldmap}
A fold map $f \co Q \to N$ is {\it oriented} if there is a chosen consistent orientation
of every fiber at their regular points, i.e., if the
kernel of the differential of the restriction $f|_R$ is an oriented bundle, where
$R$ denotes the set of regular points of $f$. 
\end{defn}

For example, a  fold map  between oriented manifolds is naturally oriented.
In this paper, we deal with fold maps into Euclidean spaces. For such a map an orientation is equivalent
to an orientation of its source manifold (we fix orientations for Euclidean spaces). 

Note that there exist oriented fold maps $f \co Q^{n+q} \to \R^{n}$ with odd $q$, which cannot be
framed in the sense of Definition~\ref{cooricuspmap} (for example for $n=3$ and $q=1$, see \cite{Sae03}).

\subsection{Existence of framed fold maps}\label{existfrfold}

We will study and recall results about the relation between 
\begin{itemize}
\item
existence of continuous fiberwise epimorphisms $TQ \oplus \ep^1_Q \to T\R^n$,
\item
having $n$ linearly independent continuous sections in  $TQ \oplus \ep^1_Q$ and
\item
existence of framed fold maps $Q \to \R^n$.
\end{itemize}

Fix the standard Riemannian metric on $\R^n$. When there is a Riemannian metric on $Q$, we always
consider the Riemannian metric on $TQ \oplus \ep^1_Q$ by defining $\ep^1_Q$ to be perpendicular to $TQ$.
%On the space $\R^n \oplus \R$ we consider the Riemannian metric 
%$\langle (x_1, x_2), (y_1, y_2) \rangle = \langle x_1, y_1 \rangle + \langle x_2, y_2 \rangle$, where
%$x_{1}, y_1 \in \R^n, x_{2}, y_2 \in \R$.
%\begin{prop}
At first observe that if there is a given Riemannian metric on $Q$, then 
having a fiberwise epimorphism $TQ \oplus \ep^1_Q \to T\R^n$ is equivalent to having $n$ linearly independent sections 
in $TQ \oplus \ep^1_Q$. 
%\end{prop}
%\begin{proof}
%Suppose $\va \co TQ \oplus \ep^1_Q \to T\R^n$ is a fiberwise epimorphism. 
%Its kernel is the bundle $\csi$ over $Q$. The Riemannian metric on $Q$ gives an orthogonal complement vector bundle to $\csi$, denote this by $\eta$.
%The restriction of $\va$ to $\eta$ is an isomorphism between $n$-dimensional vector bundles. This isomorphism gives $n$ linearly independent
%vectors in each of the fibers of $TQ \oplus \ep^1_Q$. 
%Conversely, if we have $n$ linearly independent sections of $TQ \oplus \ep^1_Q$, then 
%define a fiberwise homomorphism by mapping the linearly independent vectors in the fibers of $TQ \oplus \ep^1_Q$ 
%to the standard basis of $\R^n$. This defines a fiberwise epimorphism.
%\end{proof}

Following \cite[Lemma 3.1]{Sa1} and \cite[Lemma 3.1]{An3},
given a framed fold map $g \co Q \to \R^n$,
we will construct a fiberwise epimorphism $\va \co TQ \oplus \ep^1_Q \to T\R^n$, which will
depend on the given framed fold map $g$, a chosen Riemannian metric $\varrho$ on $Q$ and a chosen $r > 0$, where this $r$ also depends on 
$g$ and $\varrho$.

So if $g \co Q \to \R^n$ is a framed fold map and $\varrho$ is a Riemannian metric on $Q$, then let 
$r  = r(g, \varrho) > 0$ be the radius of a compact tubular neighborhood $N_r(S_g)$ of the singular set $S_g$ in $Q$.
 If we have these $g$, $\varrho$ and $r$, we define 
$$\va(g, \varrho, r) \co TQ \oplus \ep^1_Q \to T\R^n$$ in the following way.

Consider the differential $dg$ of $g$ as a homomorphism $TQ \to g^*T\R^n$
and take the commutative diagram
\begin{equation*}
\begin{CD}
TQ|_{S_g} @> i_* >> TQ @> dg >>  g^*T\R^n @> g_* >> T\R^n \\
@VVV @VVV @VVV @VVV \\
S_g @> i >> Q @> = >> Q @> g >> \R^n
\end{CD}
\end{equation*}
where $i \co S_g \to Q$ is the embedding of the singular set.
 Then $dg \circ i_*$ maps $TQ|_{S_g}$ into a codimension $1$ subbundle of $g^*T\R^n|_{S_g}$. Having the pulled-back standard Riemannian metric
on $g^*T\R^n$ we take the orthogonal complement 
$\left(\im dg \circ i_* \right)^{\perp}$ in $g^*T\R^n|_{S_g}$. This is the cokernel bundle of $dg$ over $S_g$,
which is a trivial bundle $\ep^1_{S_g}$ since $g$ is a framed fold map. 
We denote this cokernel bundle by 
$
\th_g \lra S_g.
$
Of course $\th_g \subset g^*T\R^n$ and $\th_g \cong \ep^1_{S_g}$.
Moreover we have a fiberwise monomorphism 
\begin{equation}\label{triv_monom}
g_*|_{\th_g } \co \th_g  \to T\R^n
\end{equation}
of this trivial bundle since it is a subbundle of $g^*T\R^n$, and $g^*T\R^n$ is sent by
 the pull-back homomorphism $g_* \co g^*T\R^n \to T\R^n$ fiberwise isomorphically into $T\R^n$.
Then over $S_g$ we have a fiberwise epimorphism
$$dg |_{S_g} + {\mathrm {id}}_{\th_g}  \co TQ|_{S_g} \oplus \th_g \to g^*T\R^n|_{S_g},$$
which we denote by $\tilde \va$. 
Of course $\tilde \va$ depends on 
the Riemannian metric on $\R^n$ but that is fixed so 
$\tilde \va$ depends really only on the map $g$.
We just have to extend this $\tilde \va$ somehow over the entire $TQ \oplus \ep^1_Q$ and then to compose with 
$g_*$ to get the claimed $\va$.

So take a Riemannian metric $\varrho$ on $Q$ and take $r = r(g, \varrho)$.
Let us map a point in  $N_r(S_g)$ to the closest point in $S_g$, this defines a map $p \co N_r(S_g) \to S_g$.
Then take the commutative diagram
\begin{equation*}
\begin{CD}
p^* \left( g^*T\R^n|_{S_g} \right) @> p_* >> g^*T\R^n|_{S_g}   \\
@VVV @VVV  \\
N_r(S_g) @> p >> S_g 
\end{CD}
\end{equation*}
where $p^* \left( g^*T\R^n|_{S_g} \right) $ is canonically isomorphic to 
$g^*T\R^n|_{N_r(S_g)}$ so that this isomorphism is the identity over $S_g$
since $T\R^n$ is a trivial bundle with the standard trivialization.
Then, by the isomorphism and the diagram, the pull-back by $p$ of our trivial cokernel bundle
  $\th_g$ in $g^*T\R^n|_{S_g}$ 
yields an 
$\ep^1_{N_r(S_g)}$ subbundle in $g^*T\R^n|_{N_r(S_g)}$.

Denote the bundle projection $\ep^1_{N_r(S_g)}  \to N_r(S_g)$ by $\pi$.
Then
extend $\tilde \va$ over $N_r(S_g)$ to get a fiberwise epimorphism 
$\tilde \va \co TQ|_{N_r(S_g)} \oplus \ep^1_{N_r(S_g)} \to g^*T\R^n$
 by the formula 
$$%\begin{equation}%\label{N(S_g)details_epi}
(v, w) \mapsto  dg|_{N_r(S_g)}(v) +  \al(\pi(w))  w
$$%\end{equation}
for $v \in TQ|_{N_r(S_g)}$ and $w \in \ep^1_{N_r(S_g)}  \subset g^*T\R^n|_{N_r(S_g)}$,
where 
$\al \co N_r(S_g) \to [0, 1]$ is equal to $r$ minus the distance from $S_g$ multiplied by $1/r$.
Observe that $\al(\pi(w)) = 0$ if $w$ is over the boundary of $N_r(S_g)$. We suppose that $\al$ is smooth near $S_g$.

Hence $\tilde \va$ extends continuously 
over the entire $TQ \oplus \ep^1_Q$
by applying only  $dg$ over $Q - N_r(S_g)$. So we obtain the fiberwise epimorphism 
$$\tilde \va(g, \varrho, r)  \co TQ \oplus \ep^1_Q \to g^*T\R^n,$$ 
\begin{equation}\label{fiberwiseepi}
(v, w) \mapsto 
\left\{
\begin{array}{cc}
dg(v) +  \al(\pi(w)) w & \mbox{\ \ \ \ \ if $(v, w) \in (TQ \oplus  \ep^1_Q)|_{N_r(S_g)}$},  \\
  dg (v) & \mbox{\ \ \ \ \ if $(v, w) \in (TQ  \oplus \ep^1_Q)|_{Q - N_r(S_g)}$}.
\end{array}
\right.
\end{equation}
The  restriction of $\tilde \va(g, \varrho, r)$ to $TQ$ is equal to the differential $dg$.
It is easy to see that $\tilde \va(g, \varrho, r)$ is smooth outside of a neighborhood of the boundary $\del  N_r(S_g)$. 
Finally we obtain our claimed  homomorphism by
$$
 \va(g, \varrho, r) = g_* \circ \tilde  \va(g, \varrho, r).
$$
Obviously $\va(g, \varrho, r) $  covers the map $g$.

Conversely, if we have any fiberwise epimorphism $\va  \co TQ \oplus \ep^1_Q \to T\R^n$, then
there is a framed fold map $g \co Q \to \R^n$. More precisely, we have the following.

\begin{thm}[Theorem~3.2 in \cite{An3}]\label{Andothm}
Let $n \geq 2$ and $q \geq 0$.
Assume there is a fiberwise epimorphism $\psi \co TQ^{n+q} \oplus \ep^1_Q \to T\R^n$ over a continuous map $g \co Q^{n+q}  \to \R^n$.
Then there is a framed fold map $f \co Q^{n+q} \to \R^n$ homotopic to $g$.
\end{thm}

\begin{rem}\label{framinghomot}
If the bundle $TQ \oplus \ep^1_Q$ has $n$ linearly independent sections 
$e_1, \ldots, e_n$
and there is a given Riemannian metric $\varrho$ on $Q$,
then mapping the orthogonal complement of the subspace spanned by the sections to $0$ and mapping the sections
to the standard framing of $T\R^n$ we get a fiberwise epimorphism $\psi \co TQ \oplus \ep^1_Q \to T\R^n$ over some map $g \co Q \to \R^n$.
By Theorem~\ref{Andothm}, from $\psi$ we get a framed fold map $f \co Q \to \R^n$
homotopic to $g$. 
This framed fold map $f$ also gives a fiberwise epimorphism $\va(f, \varrho, r) \co TQ \oplus \ep^1_Q \to T\R^n$ for some $r >0$
and hence $n$ linearly independent sections 
$e_1', \ldots, e_n'$ in $TQ \oplus \ep^1_Q$.
As the proof of \cite[Theorem~3.2]{An3} shows, we get $f$ by 
constructing from $\psi$ an appropriate section $s$ of $Q$ into the $2$-jet space and by
 applying \cite[Theorem~2.1]{An3} which gives the homotopy from $g$ to $f$. 
This homotopy is the result of a homotopy in the formal $2$-jet space,
which gives a homotopy of  fiberwise epimorphisms from $\psi$ to $\va(f, \varrho, r)$ as well, which gives 
a homotopy from $e_1, \ldots, e_n$ to $e_1', \ldots, e_n'$.
\end{rem}

We will use an easy modification of the previous statement for the ``relative case'' as follows.
As usual, for a smooth map $f \co Q \to N$ and $x \in Q$, $y \in N$, 
we express a $2$-jet in $J^2_{x, y}(Q, N)$ as
a pair $(a, b)$ where
$a \in \mathrm{Hom}(T_xQ, T_yN)$, $b \in \mathrm{Hom}(S^2(T_xQ), T_yN)$ and
$S^2(T_xQ)$ denotes the $2$-fold symmetric product of $T_xQ$, see \cite[pages 32--33]{An3}.

\begin{thm}\label{relexifold}
Let $k \geq 2$ and $m \geq k$.
Let $\varrho$ be a Riemannian metric on the $m$-dimensional manifold $W$ and 
let $C$ be a closed subset of  $W$.
Let 
$f \co W \to \R^k$ be a continuous map such that the restriction of $f$ 
to a neighborhood $U$ of $C$ is  
smooth and has only definite fold singularities. Let $\psi \co TW \oplus \ep^1_W \to T\R^k$ be 
a fiberwise epimorphism over $f$, and suppose
that $\psi$ is equal to $\va(f, \varrho, r)$  for some $r>0$ over $U$. 
Then, there exists a framed fold map $g \co W \to \R^k$,
which coincides with $f$ on $C$.
\end{thm}
\begin{proof}
Let $\csi$ be the kernel of $\psi$. Then in the same way as in the proof of \cite[Theorem~3.2]{An3}
we obtain a $(k-1)$-dimensional manifold $V$ in $W$ such that
$\csi|_p \subset T_pW$ exactly at the points $p \in V$.
It follows that for the singular set $S$ of $f|_U$ we have $S \subset V$.
Also  we have that the rank of $\psi|_{T_pW}$ is equal to $k$ for $p \in W - V$ and it is equal to 
$k-1$ for  $p \in V$. Clearly $\psi|_{TW}$ induces a homomorphism 
$\Psi \co TW \to f^*(T\R^k)$ by  pulling back by $f$.
Then 
for $p \in U$ the restriction $\Psi|_{T_pW}$ is equal to 
the differential $df \co T_pW \to f^*(T\R^k)$ because $\psi|_{T_pW} = f_* \circ df|_{T_pW}$.

Observe that the cokernel bundle of $\Psi$ in $f^*(T\R^k)|_V$  at the points of $V$ is the trivial bundle:
the sequence
$$
0 \lra \csi|_V \lra TW|_{V} \xrightarrow{\text{$\Psi |_V$}} f^*(T\R^k)|_V \lra {\mathrm {coker}} \thinspace \Psi |_V \lra 0
$$
\begin{equation*}
\begin{CD}
0 @>>> \csi|_V @>>> TW|_{V} @>  \Psi |_V >> f^*(T\R^k)|_V @>>> {\mathrm {coker}} \thinspace \Psi |_V @>>> 0
\end{CD}
\end{equation*}
is obviously exact and then by the Whitney product formula 
we have
$$
w_1 ( {\mathrm {coker}} \thinspace \Psi |_V )  + w_1 ( TW|_{V} ) - w_1 (  \csi|_V ) = w_1 ( f^*(T\R^k)|_V  ). 
$$
This implies that 
\begin{equation}\label{coker_triv}
w_1 ( {\mathrm {coker}} \thinspace \Psi |_V )  =  w_1 (  \ep^1_V ),
\end{equation}
because $$w_1 ( f^*(T\R^k)|_V  ) + w_1 (  \csi|_V ) - w_1 ( TW|_{V} )  = w_1 (  \ep^1_V )$$ 
again by the Whitney product formula
since the sequence 
$$0 \lra  \csi|_V \lra TW|_{V} \oplus  \ep^1_V \lra  f^*(T\R^k)|_V \lra 0$$
is exact. So by (\ref{coker_triv}) the bundle ${\mathrm {coker}} \thinspace \Psi |_V$ is trivial.
Now we want to define a non-singular symmetric map $\be \co \csi|_V \otimes \csi|_V \to \ep^1_V$ whose
target $\ep^1_V \subset f^*(T\R^k)|_V$ is just the cokernel bundle of $\Psi$ at the points of $V$. 
For the singular set $S \subset V$ we have already such a map: the positive definite symmetric form 
 given by the definite fold singularities of $f_U$. We denote this by 
 $$\be \co \csi|_S \otimes \csi|_S \to \ep^1_S.$$
 This extends to the entire $\csi|_V \otimes \csi|_V$ as a non-singular symmetric map because
 the symmetry group of the definite fold singularity is the full orthogonal group. Let us denote  this extension by $\be$.
Now this $\be$ can be extended to a homomorphism $\tilde \be \co S^2(TW) \to f^*(T\R^k)$ 
%where $S^2(TW)$ denotes the $2$-fold symmetric product of $TW$, 
so that over $U$ the pair of homomorphisms
$\left( \Psi, \tilde \be \right)$ is equal to the $2$-jet of $f$.

So if we define the section $s \co W \to \Omega^{m-k+1, 0}(W, \R^k)$, where 
$\Omega^{m-k+1, 0}(W, \R^k)$ denotes the union of regular and fold jets in the $2$-jet space $J^2(W, \R^k)$,
 as
$$s(p) = \left(p, f(p), \Psi_p, \tilde \be_p \right),$$
then 
by applying  \cite[Theorem~2.1]{An3} we get the statement.
\end{proof}

\section{Cobordism of maps}\label{kob}

\subsection{Cobordism of framed and oriented maps}

\begin{defn}[Cobordism]\label{cobdef} 
Two fold maps $f_i \co Q_i \to \R^n$, $i=0,1$,  
of closed $({n+q})$-dimensional manifolds $Q_i^{n+q}$  are
%into an $n$-dimensional manifold $N^n$ are  
{\it cobordant} if there exists a fold map 
$$F \co X \to \R^n \times [0,1]$$ of a compact $(n+q+1)$-dimensional 
manifold $X$ such that
\begin{enumerate}[(i)]
\item
$\del X = Q_0 \amalg Q_1$ and %(resp. $Q_0^{n+q} \amalg -Q_1^{n+q}$) and
\item
${F|}_{Q_0 \x [0,\ep)}=f_0 \x
{\mathrm {id}}_{[0,\ep)}$ and ${F|}_{Q_1 \x (1-\ep,1]}=f_1 \x 
{\mathrm {id}}_{(1-\ep,1]}$, where 
$Q_0 \x [0,\ep)$
 and $Q_1 \x (1-\ep,1]$ are small collar neighborhoods of $\del X$ with the
identifications $Q_0 = Q_0 \x \{0\}$ and $Q_1 = Q_1 \x \{1\}$. 
%where $\ep$  refers to a small positive number. 
\end{enumerate}
We call the map $F$ a {\it cobordism} between $f_0$ and $f_1$. 

When the fold maps $f_i \co Q_i \to \R^n$, $i=0,1$,  are oriented,
we say that they are {\it oriented cobordant} (or shortly {\it cobordant} if it is clear from the context) if
they are cobordant in the above sense via an oriented fold map $F$, such that the orientations are compatible on the boundary of $X$.
\end{defn} 

This clearly defines an equivalence relation on the set of  fold maps
of closed $({n+q})$-dimensional manifolds into $\R^n$. The equivalence classes are called {\it cobordism classes}.
We denote 
 the set of  cobordism classes of  fold maps 
 %of closed $({n+q})$-dimensional manifolds into $\R^n$
by $\CC ob_{{}}^{}(n,q)$.
%When $N^n=\R^n$, we denote it by $\CC ob_{f}^{}(n+q,-q)$ (resp. $\CC ob_{f}^{O}(n+q,-q)$).
For the oriented version an upper index ``$O$'' applies and we write   $\CC ob_{{}}^{O}(n,q)$.
By taking disjoint union of maps, the sets $\CC ob(n,q)$ and $\CC ob^{O}(n,q)$ become  groups as one can see easily.

\begin{defn}[Framed cobordism]\label{framedcobdef}
Two framed fold maps $f_i \co Q_i \to \R^n$, $i = 0, 1$,   
of closed $({n+q})$-dimensional manifolds $Q_i$  
%into an $n$-dimensional manifold $N^n$ are  
are {\it  framed cobordant} if 
they are cobordant in the sense of Definition~\ref{cobdef}
by a framed fold map $F \co X \to \R^n \times [0,1]$ 
such that the orientation of the normal bundle of the immersion 
$F|_{S_F^{}} \co S_F^{} \to \R^n \times [0,1]$ restricted 
to $\del X \cap S_F^{}$ coincides with that of 
 the immersions $f_i|_{S_{f_i}^{}} \co S_{f_i}^{} \to \R^n \times \{ i \}$,  $i=0,1$.
Similarly two oriented framed fold maps are \emph{oriented framed cobordant}
if they are framed cobordant so that this cobordism is also an oriented cobordism.
\end{defn}

We denote the corresponding framed cobordism groups by  $\CC ob_{fr}^{}(n,q)$ and $\CC ob_{fr}^{O}(n,q)$.
For  $q$ even the group 
$\CC ob_{fr}^{}(n,q)$ is naturally isomorphic to
$\CC ob_{{}}^{}(n,q)$ and 
the group 
$\CC ob_{fr}^{O}(n,q)$ is naturally isomorphic to
$\CC ob_{{}}^{O}(n,q)$ by the natural forgetful map from the framed cobordism group to the unframed one.

\subsection{Cobordism invariants of fold maps}\label{geomcobinv}

We introduced and used geometric invariants of cobordisms of fold maps \cite[Section~2]{Kal7}, namely the homomorphisms
\[
\csi_{{\la}}  \co \CC ob_{{}}^{}(n,q) \to \imm \left({\ep^1_{B(O(\la) \x O(q+1-\la))}},n\right)
\]
for $0 \leq \la < (q+1)/2$ and
\[
\csi_{(q+1)/2}  \co \CC ob_{{}}^{}(n,q) \to \imm({l^1},n)
\]
for $q$ odd and $\la=(q+1)/2$, where $l^1$ is an appropriate line bundle. 
These homomorphisms are defined as follows.
Restricting the fold map $f \co Q^{n+q} \to N^n$  to a tubular neighborhood of its index $\la$
fold singular set $S_{\la}^{}(f)$  we get a bundle over $S_{\la}^{}(f)$ with fiber the mapping
$$\va \co (x_1,\ldots,x_{q+1}) \mapsto  \sum_{i = 1}^{\la}
-x_i^2  + \sum_{i = \la + 1}^{q+1} x_{i}^2.$$

%By \cite{Jan, Wa} this bundle is a locally trivial bundle 
% in a sense with a fiber $\phi$
% and an appropriate group of automorphisms of $\phi$ as structure group.
 By \cite{Jan, Wa}
the structure group of this bundle can be reduced to a maximal compact subgroup, namely to the 
group $O(\la) \x O(q+1-\la)$ in the case of $0 \leq \la < (q+1)/2$ and to the group 
generated by the group $O\left(\frac{q+1}{2}\right) \x O\left(\frac{q+1}{2}\right)$ and the transformation 
$ T = 
\begin{pmatrix}
0 & I_{(q+1)/2} \\ I_{(q+1)/2} & 0 
\end{pmatrix}
$ in the case of $q$ odd and $\la = (q+1)/2$, see, for example, \cite{Sa1}.
We denote this latter group by $\left\langle O\left(\frac{q+1}{2}\right) \x O\left(\frac{q+1}{2}\right), T \right\rangle$.

The $\va$-bundle over $S_{\la}^{}(f)$ has a ``source'' bundle and also a ``target'' bundle part over $S_{\la}^{}(f)$. 
This results from the source $\R^{q+1}$ and the target $\R$ of $\va$ since
the structure group of $\va$ acts on these. So we have a $(q+1)$-dimensional vector bundle and also a $1$-dimensional vector bundle over 
$S_{\la}^{}(f)$ and $\va$ maps fiberwise between them.

It follows that this $1$-dimensional vector bundle over $S_{\la}^{}(f)$, where
$0 \leq \la \leq (q+1)/2 $, 
can be induced from the trivial line bundle
$$\ep^1 \to B(O(\la) \x O(q+1-\la))$$ for $\la \neq (q+1)/2$ and from an appropriate line bundle
$$l^1 \to B\left\langle O\left(\frac{q+1}{2}\right) \x O\left(\frac{q+1}{2}\right), T \right\rangle$$ 
for $q$ odd and $\la = (q+1)/2$.

Now, restricting the fold map $f$ to its fold singular set $S_{\la}^{}(f)$ of index $\la$ we get an immersion and
 the homomorphisms $\csi_{{\la}}$ and $\csi_{(q+1)/2}$
map
 the cobordism class $[f]$ to the cobordism class of this immersion   with normal bundle induced from
the line bundle
$\ep^1 \to B(O(\la) \x O(q+1-\la))$ or $l^1 \to B\left\langle O\left(\frac{q+1}{2}\right) \x O\left(\frac{q+1}{2}\right), T \right\rangle$, respectively.

\subsubsection{Invariants of oriented fold maps}
For oriented fold maps, we have the analogous statements 
but we have to consider the subgroup of the automorphisms of $\va \co \R^{n+q} \to \R$ whose elements act   preserving 
the orientations of $\R^{n+q}$ and $\R$ simultaneously or reversing the orientations simultaneously.
That is,
we consider the subgroup $S(O(\la) \x O(q+1-\la))$ of orientation preserving transformations of the group
$O(\la) \x O(q+1-\la)$ 
and the trivial line bundle $\ep^1 \to BS(O(\la) \x O(q+1-\la))$
in the case of $0 \leq \la < (q+1)/2$,
and the appropriate subgroup denoted by $S{\left\langle O\left(\frac{q+1}{2}\right) \x O\left(\frac{q+1}{2}\right), T \right\rangle}$
of the group $\left\langle O\left(\frac{q+1}{2}\right) \x O\left(\frac{q+1}{2}\right), T \right\rangle$
and the corresponding line bundle
$\tilde l^1 \to BS{\left\langle O\left(\frac{q+1}{2}\right) \x O\left(\frac{q+1}{2}\right), T \right\rangle}$ 
in the case of $q$ odd and  $\la = (q+1)/2$.
So in the case of oriented fold maps, %(e.g., in the case  of oriented manifolds $Q^{n+q}$ and $N^n$), 
we have the  homomorphisms 
\[
\csi_{{\la}}^{O}  \co \CC ob_{{}}^{O}(n,q) \to \imm\left({\ep^1_{BS(O(\la) \x O(q+1-\la))}},n\right)
\]
for $0 \leq \la < (q+1)/2$ and
\[
\csi_{(q+1)/2}^{O}  \co \CC ob_{{}}^{O}(n,q) \to \imm({\tilde l^1},n)
\]
for $q$ odd and $\la=(q+1)/2$, just like in the case of non-oriented fold maps.
We used these homomorphisms in \cite{Kal4, Kal7} to describe cobordisms of fold maps. 

\subsubsection{Invariants of framed fold maps}
For $0 \leq \la < (q+1)/2$,
clearly there is the similar framed cobordism invariant
\[
\csi_{{\la}}  \co \CC ob_{fr}^{}(n,q) \to \imm\left({\ep^1_{B(O(\la) \x O(q+1-\la))}},n\right).
\]
For framed fold maps when $q$ is odd and $\la = (q+1)/2$ we have to consider only that largest subgroup of 
$\left\langle O\left(\frac{q+1}{2}\right) \x O\left(\frac{q+1}{2}\right), T \right\rangle$ whose elements act trivially on the target $\R$ of $\va$.
%If $(q+1)/2$ is odd, then 
This is $O\left(\frac{q+1}{2}\right) \x O\left(\frac{q+1}{2}\right)$ because $T$ acts as multiplication by $-1$ on the target $\R$ of $\va$. So
we have the homomorphism 
\[
\csi_{(q+1)/2}  \co \CC ob_{fr}^{}(n,q) \to \imm\left({\ep^1_{B\left(O\left(\frac{q+1}{2}\right) \x O\left(\frac{q+1}{2}\right)\right)}},n\right).
\]
%If $(q+1)/2$  is even, then all elements of $O({}rac{q+1}{2}) \x O({}rac{q+1}{2})$  and also $T$ act trivially so in this case 
%the line bundle $l^1$ is trivial and the framing gives a trivialization hence we have
%\[
%\csi_{(q+1)/2}  \co \CC ob_{{}, fr}(n,q) \to \imm({\ep^1_{B\langle O({}rac{q+1}{2}) \x O({}rac{q+1}{2}), T \rangle}},n).
%\]
It is easy to get the analogous homomorphisms for oriented framed fold maps as well.

Since the group $\imm ( \ep^1_X, n)$, where $X$ is a space, is isomorphic to 
$\pi_{n-1}^s \oplus \pi_{n-1}^s(X)$, all these homomorphisms $\csi_{\la}$, where $0 \leq \la \leq (q+1)/2$, map 
from a framed cobordism group into a direct sum of stable homotopy groups.
The first coordinate map of $\csi_{\la}$ mapping to  $\pi_{n-1}^s$ will be denoted by 
$$t_{\la} \co \CC ob_{fr}^{}(n,q) \to \pi_{n-1}^s$$ and
the second will be
$$\tau_{\la} \co \CC ob_{fr}^{}(n,q) \to \pi_{n-1}^s(BG),$$
 for the suitable subgroup $G$ of the orthogonal group $O(q+1)$.
 
 When $n=2$, 
the sum for $\la \geq 0$ %and for $\la \geq 1$, respectively, 
of these homomorphisms will be used, i.e.\ take
$$\sum_{0 \leq \la \leq (q+1)/2} t_{\la}$$
and denote it by $t$ and
take 
$$\sum_{0 \leq \la \leq (q+1)/2} \tau_{\la}$$
and denote it by $\tau$.

So we have two homomorphisms $$t \co \CC ob_{fr}^{}(2,q) \to \Z_2,$$
because $\pi_{1}^s = \Z_2$
and $$\tau \co \CC ob_{fr}^{}(2,q) \to \Z_2 \oplus \Z_2$$
because $\pi_{1}^s(BG) = \Z_2 \oplus \Z_2$ for any $1 \leq \la \leq (q+1)/2$ and $G$ at hand, and in
the case of $\la = 0$ we consider  $\pi_{1}^s(B(O(0) \x O(q+1))) = \{ 0 \} \oplus \Z_2$, which we consider as a subgroup of $\Z_2 \oplus \Z_2$.
We will denote the two components of $\tau$ by $\tau^1$ and $\tau^2$.
 
The analogous homomorphisms for the case of oriented framed maps are denoted the same way. In that case we have
$$t \co \CC ob_{fr}^{O}(2,q) \to \Z_2$$ and
$$\tau \co \CC ob_{fr}^{O}(2,q) \to  \Z_2$$
because then $\pi_{1}^s(BG) =  \Z_2$ for any $1 \leq \la \leq (q+1)/2$ and $G$ at hand, and 
in the case of $\la = 0$ we have  $\pi_{1}^s(BS(O(0) \x O(q+1))) = \{ 0 \}$, which we consider as  a subgroup of $\Z_2$.

To simplify the notation, if $f$ is a fold map,  we will refer to $t([f])$ and $\tau([f])$ as $t(f)$ and $\tau(f)$, respectively.
When $n=2$, the value $t(f)$ is just the number of double points mod $2$ of the immersion into $\R^2$ of the singular set of the generic fold map $f$.

\begin{rem}\label{nfrori}
Clearly in the case of non-framed (but possibly oriented or non-oriented) fold maps into $\R^2$ of {\it even} codimension $q$ we have the analogous 
homomorphisms $t$ and $\tau$
because 
such fold maps and their cobordisms can be naturally framed.

We will  consider non-framed oriented fold maps to $\R^2$ of {\it odd} codimension $q$ as well.
In that case we also have the analogous homomorphisms $t$ and $\tau$ mapping into $\Z_2$ since
for $t_{\la}$ and  $\tau_{\la}$, where $0 \leq \la < (q+1)/2$,   everything works the same way as in the framed case. Moreover
an easy computation shows that the target of the homomorphism
\[
\csi_{(q+1)/2}^{O}  \co \CC ob_{{}}^{O}(2,q) \to \imm({\tilde l^1},2)
\]
is also $\Z_2 \oplus \Z_2$, which corresponds to the number of double points mod $2$ and the twisting of 
the normal bundle of the 
the index $(q+1)/2$ fold singularities in the source manifold, just like in the case
of framed fold maps.
\end{rem}

\section{Cobordism of manifolds with stable framings}\label{ujcob}

\subsection{Stably framed manifolds and their cobordisms}

\begin{defn}[Stably $(n-1)$-framed manifolds]
For $n>0, q \geq 0$
an $(n+q)$-dimensional manifold $Q$
is {\it stably $(n-1)$-framed} if the vector bundle $TQ \oplus \ep^1_{Q}$
has $n$ sections that are linearly independent at every point of $Q$ (shortly, we say that the vector bundle $TQ \oplus \ep^1_{Q}$
 has $n$
independent sections).
\end{defn}

\begin{defn}[Stably $(n-1)$-framed cobordism]\label{stabframecob}
Let $Q_i$ be closed (oriented) stably $(n-1)$-framed $(n+q)$-dimensional manifolds, i.e.,
the vector bundles $TQ_i \oplus \ep^1_{Q_i}$
have $n$ independent sections $e_i^1, \ldots, e_i^{n}$ ($i=0,1$). 
We say that the manifolds $Q_0$ and $Q_1$ 
are {\it stably (oriented) $(n-1)$-framed cobordant} if 
\begin{enumerate}[(i)]
\item
they are (oriented) cobordant in the
usual sense by a compact (oriented) $(n+q+1)$-dimensional manifold $W$,
\item
the vector bundle $TW \oplus \ep^1_W$
has $n+1$ independent sections $f^1, \ldots, f^{n+1}$,
\item
the sections $f^j$, $j = 1, \ldots, n$,
restricted to the boundary $Q_0 \amalg Q_1$ of $W$ 
coincide with the sections
$e_i^j$  ($j=1, \ldots, n$ and $i=0,1$).
\item
the section $f^{n+1}$ restricted to the boundary part $Q_0$ (resp.\ $Q_1$) of $W$ 
coincides with an inward (resp.\ outward) normal section of $\del W$.
\end{enumerate}
\end{defn}

We denote the set of stably $(n-1)$-framed cobordism classes
of closed  stably $(n-1)$-framed $(n+q)$-dimensional manifolds 
 by 
$\CC_{n+q}^{}(n)$ (and by $\CC_{n+q}^{O}(n)$ in the oriented case) which is an abelian group with the disjoint union as operation.

%By using \cite[Lemma 3.1]{Sa1} (see also \cite[Lemma 3.1]{An3}), 
We obtain  homomorphisms 
$$\si_{n,q} \co \CC ob_{fr}(n,q) \to   \CC_{n+q}(n)$$
and
$$\si^{O}_{n,q} \co \CC ob_{fr}^{O}(n,q) \to   \CC_{n+q}^{O}(n),$$
which map a cobordism class of a framed fold map $g \co Q^{n+q} \to \R^n$ 
to the stably $(n-1)$-framed cobordism class of the
source manifold $Q^{n+q}$ precisely as follows.

In Section~\ref{existfrfold} we constructed a fiberwise epimorphism
$$\va(g, \varrho, r) \co TQ \oplus \ep^1_Q \to T\R^n$$
 from a given framed fold map $g \co Q^{n+q} \to \R^n$.
%At first, we will construct a fiberwise epimorphism $\va \co TQ \oplus \ep^1_Q \to T\R^n$, which will
%depend on the given framed fold map $g$, a chosen Riemannian metric $\varrho$ on $Q$ and a chosen $r > 0$, where this $r$ also depends on 
%$g$ and $\varrho$. 
We will define the cobordism group of fiberwise epimorphisms of this form which we will denote by
$\EE(n, q)$
and
give a homomorphism from the group $\CC ob_{fr}(n,q)$ to $\EE(n, q)$. Then we will define 
a homomorphism from $\EE(n, q)$ to the group $\CC_{n+q}(n)$. The composition 
$$\CC ob_{fr}(n,q) \to \EE(n, q) \to \CC_{n+q}(n)$$
of these two 
homomorphisms will be $\si_{n,q}$. 

\begin{defn}
Let $Q_0$ and $Q_1$ be closed $(n+q)$-dimensional manifolds.
Let $$\psi_i  \co TQ_i \oplus \ep^1_{Q_i} \to T\R^n,$$ $i = 0, 1$, be
fiberwise epimorphisms. We say that $\psi_0$ and $\psi_1$ are {\it cobordant} if
\begin{enumerate}[\rm (i)]
\item
there is a compact $(n+q+1)$-dimensional manifold $W$ such that $\del W = Q_0 \amalg Q_1$,
\item
there is a fiberwise epimorphism $\Psi \co TW \oplus \ep^1_W \to T\left( \R^n \x [0,1] \right)$ and
\item
for $i = 0, 1$ the bundle homomorphism $\Psi$ restricted to $TQ_i \oplus \ep^1_{Q_i}$
maps into $T\left( \R^n \x \{ i \} \right)$ and it is equal to $\psi_i$.
\end{enumerate}
The naturally resulting cobordism group is denoted by $\EE (n, q)$.
\end{defn}

We have the following (we use the notations of  Section~\ref{existfrfold}).
\begin{lem}
Let $g \co Q^{n+q} \to \R^n$ be a framed fold map.
Let $\varrho$ be a Riemannian metric on $Q$ and let $r = r(g, \varrho)$.
\begin{enumerate}[\rm (i)]
\item
If $0 < r' < r$, then $\va(g, \varrho, r')$ is cobordant to $\va(g, \varrho, r)$. 
\item
If $\varrho'$ is another Riemannian metric on $Q$ and
 $r' = r'(g, \varrho')$, then
there is a positive $r'' < \min {r}{r'} $ such that 
$\va(g, \varrho, r'')$ and $\va(g, \varrho', r'')$ are cobordant.
\item
If $(Q',  \varrho')$ is another Riemannian manifold and
$g' \co Q' \to \R^n$ is another framed fold map such that
$g$ and $g'$ are framed cobordant, then
there is a positive $$r'' < \min{r(g, \varrho)}{r'(g', \varrho')}$$ such that
$\va(g, \varrho, r'')$ and $\va(g', \varrho', r'')$ are cobordant.
\end{enumerate}
\end{lem}
\begin{proof}
We get (i) by continuously modifying along $Q \x \{t\}$, $t \in [0,1]$, the radius $r$ 
of the tubular neighborhood $N_r(S_g)$ 
until we get $r'$ 
in the manifold $Q \x [0,1]$ equipped with the Riemannian metric equal to the direct sum of $\varrho$ and the standard metric on $[0,1]$.
We also modify the function $\al \co N_r(S_g) \to [0,1]$ such that it stays smooth near $S_g$.
The proof of (ii) and (iii) is an easy exercise in constructing Riemannian metrics 
on $Q \x [0,1]$ in the case of  (ii) and on 
a cobordism $W$ in the case of (iii).
\end{proof}

This lemma implies that the construction in Section~\ref{existfrfold}  induces 
 a homomorphism from 
the cobordism group of framed fold maps $\CC ob_{fr}(n,q)$
to the cobordism group $\EE (n, q)$.

Now, we want
a homomorphism
from $\EE (n, q)$ to the cobordism group of stably $(n-1)$-framed manifolds.
The kernel of a fiberwise epimorphism $\va \co TQ \oplus \ep^1_Q \to T\R^n$
has an orthogonal complement if there is a given Riemannian metric on $Q$.
Then this orthogonal complement has $n$ frames since it is mapped isomorphically onto $T\R^n$.
So we have the stably $(n-1)$-framed manifold $Q$ and 
the image of the cobordism class $[\va]$ is defined to be the stably framed cobordism class of $Q$.
%its stably framed cobordism class 
%is defined to be the image of the cobordism class $[\va]$.
But again we have to prove that these $n$ frames  do not depend on the  Riemannian metric on $Q$ up to cobordism. 

\begin{lem}
We have the following.
\begin{enumerate}[\rm (i)]
\item
If $\varrho_1$ and $\varrho_2$ are Riemannian metrics on $Q$, then
the two stably $(n-1)$-framed manifolds obtained from the fiberwise epimorphism $TQ \oplus \ep^1_Q \to T\R^n$ depending
on  $\varrho_1$ and $\varrho_2$
are cobordant.
\item
If the fiberwise epimorphisms $TQ_i \oplus \ep^1_{Q_i} \to T\R^n$, $i = 0, 1$, are cobordant, then for some Riemannian metrics on $Q_0$ and $Q_1$ the obtained 
stably $(n-1)$-framed manifolds $Q_0$ and $Q_1$ are cobordant.
\end{enumerate}
\end{lem}
\begin{proof}
The proof is an easy exercise in extending Riemannian metrics 
on $Q \x [0,1]$ in the case of (i)  and constructing Riemannian metrics on 
a cobordism $W$ in the case of (ii).
\end{proof}
All these arguments imply that we have the well-defined homomorphisms
$$\CC ob_{fr}(n,q) \to \EE(n, q)\mbox{\ \ \ \  and \ \ \ \ \ }\EE(n, q) \to \CC_{n+q}(n).$$

\begin{defn}
We denote by $\si_{n,q}$ the composition 
$$\CC ob_{fr}(n,q) \to \EE(n, q) \to \CC_{n+q}(n).$$
Analogously, in the 
 oriented case   we have the
well-defined homomorphism 
$$\si^{O}_{n,q} \co \CC ob^O_{fr}(n,q) \to  \CC^O_{n+q}(n).$$
\end{defn}

%with the stable framing obtained by 
%\cite[Lemma 3.1]{Sa1}\footnote{In \cite[Lemma 3.1]{Sa1} $n$ independent sections for 
%$TQ^{n+q} \oplus \ep^1_{Q^{n+q}}$ are constructed, which are unique up to stably framed cobordism.}.

\begin{prop}\label{stabfrcobsurj}\label{homsurj}
The homomorphisms $\si_{n,q}$ and $\si^{O}_{n,q}$ are surjective.
\end{prop}
\begin{proof}
Let us take a cobordism class $\omega$ in $\CC_{n+q}(n)$ represented by a stably $(n-1)$-framed manifold $Q$.
Then by Remark~\ref{framinghomot} there is a framed fold map $f \co Q \to \R^n$
such that the stable $(n-1)$-framing given by $f$ is homotopic to 
the stable $(n-1)$-framing which was given originally on $Q$.
This homotopy yields a cobordism between the two stable $(n-1)$-framings. 
Hence $\si_{n,q}$ maps the cobordism class of the framed fold map $f$ 
to the cobordism class $\omega$. The proof for $\si^{O}_{n,q}$ is 
similar.
\end{proof}

%\begin{rem}
%Note that 
In fact \cite{Ko} deals extensively with the groups  $\CC_{n+q}^{}(n)$ and $\CC_{n+q}^{O}(n)$, 
which are naturally isomorphic through stabilization to cobordism groups denoted in \cite{Ko}
 by ${\mathfrak {N}}_{n+q}(n-1,n-1)$ and $\Omega_{n+q}(n-1,n-1)$, respectively, and Koschorke
computes them for low $n$ (see 
\cite[Theorem~6.6, Proposition~7.17 and Theorems~12.1, 12.8, 19.39, 19.40, 19.41]{Ko}).

For example, the group ${\mathfrak {N}}_{m}(k,k)$ is defined to be 
the cobordism group of closed $m$-dimensional manifolds admitting $k$ linearly independent
vector fields, where a cobordism $W$ between such manifolds 
is a compact $(m+1)$-dimensional manifold as usual but $W$ also admits  $k+1$ linearly independent 
vector fields extending the given $k$ vector fields on its boundary and %as expected 
the $(k+1)$th vector field on $W$ corresponds to the inward and outward normal vectors on its boundary.
For details, see \cite{Ko}.

%\end{rem}

\begin{defn}\label{eveneulercob}
Let ${\mathfrak {N}}_{m}^{2|\chi}$ and $\Omega_{m}^{2|\chi}$ denote the kernels of the homomorphism 
$$w_m \co \mathfrak {N}_{m} \to \Z_2$$ and $$w_m \co \Omega_{m} \to \Z_2,$$ respectively, where for a closed $m$-dimensional 
manifold $M^m$, $w_m([M^m])$ is the Stiefel-Whitney number for the
$m$th Stiefel-Whitney
class of $M^m$.
In other words ${\mathfrak {N}}_{m}^{2|\chi}$ (resp.\ $\Omega_{m}^{2|\chi}$) is the 
 cobordism group (resp.\ oriented cobordism group) of manifolds of even Euler characteristic.
\end{defn}

%\begin{exa}
For example, the group $\Omega_4^{2|\chi}$ is isomorphic to $\Z$ and it is generated by the class $[\CP^2 \# \CP^2]$.
%\end{exa}
By \cite[Proposition~7.17 and Theorems~12.8, 19.39]{Ko}, 
we have the following.

\begin{prop}\label{koschizom}
For $k \geq 1$, we have
\begin{enumerate}[\rm (i)]
\item
$\CC_{k+1}^{}(2) \cong  \Z_2 \oplus {\mathfrak {N}}_{k+1}^{2|\chi}$,
%$\CC_{4k-1}^{O}(1) = \Omega_{4k-1}$,
%\item
%$\CC_{4k-1}^{O}(2) \cong \Omega_{4k-1}$,
\item
$\CC_{4k}^{O}(2) \cong \Omega_{4k}^{2|\chi}$ and the isomorphism is induced by the forgetful map $\CC_{4k}^{O}(2) \to \Omega_{4k}$.
%\item
%$\CC_{4k+1}^{O}(2) \cong \Omega_{4k+1} \oplus \Z_2^2$,
%\item
%$\CC_{4k+2}^{O}(2) \cong \Omega_{4k+2} \oplus \Z_2$
%and the forgetful map $\CC_{4k+2}^{O}(2) \to \Omega_{4k+2}$ has kernel $\Z_2$.
%$\CC_{4k}^{O}(3) = \Omega_{4k}^{4|\si}$, and
%\item
%$\CC_{4k+4}^{O}(4) \cong \Omega_{4k+4}^{8|\si}$. 
\end{enumerate}
\end{prop}

\subsection{The singular set of the stably framed source manifold}

It is important to know what $\si_{n,q}$ does if we watch it through the isomorphism stated in (i) of Proposition~\ref{koschizom}.
%If $q$ is even, then $\CC ob_{fr}(n,q)$ is the same as $\CC ob(n,q)$ since even codimensional fold maps
%can be framed naturally. 
The next statement explains in geometric terms  the homomorphism $\chi''$ defined in \cite[page 130 (12.5)]{Ko},
which plays an important role in the theory of framed bordisms.
Following \cite[Definition~2.1 and pages~26--27]{Ko} we will denote by 
$\Omega_1(point; \ep^1)$, where $\ep^1$ is the trivial line bundle over a point, the first normal bordism group of the point with trivial coefficients. This group is isomorphic to 
the stable homotopy group $\pi_1^s$, see \cite[equation (2.2) on page 27]{Ko}.
Recall that $\pi_1^s \cong \Z_2$ and hence $\Omega_1(point; \ep^1) \cong \Z_2$. The generator is 
represented by a circle embedded in a high dimensional Euclidean space with trivialized normal bundle such that the trivialization is twisted once as we move along 
the circle.

\begin{prop}\label{koschizomexplained}
Let $q \geq 0$.
Under the isomorphism $\CC_{2+q}(2) \cong  \Z_2 \oplus {\mathfrak {N}}_{2+q}^{2|\chi}$ obtained in 
\cite[Theorem~12.8]{Ko},  the homomorphism
$$\si_{2,q} \co \CC ob_{{fr}}(2,q) \to    \Z_2 \oplus {\mathfrak {N}}_{q+2}^{2|\chi}$$
is the map $$[g \co M^{2+q} \to \R^2] \mapsto \left([g|_{S_g}], [M^{2+q}]\right),$$ where
$[M]$ is the cobordism class of $M$  and $[g|_{S_g}] \in \Z_2$ is the cobordism class
of the immersion $g|_{S_g}$ of the fold singular set of $g$ into $\R^2$.
\end{prop}
\begin{proof}
By definition the homomorphism  $\si_{2,q}$ maps a fold cobordism class $[g \co M^{2+q} \to \R^2]$ to the cobordism class of the stably $1$-framed source manifold $[M]$ 
in $\CC_{2+q}(2)$.
By \cite[Proposition~7.17]{Ko} the group $\CC_{2+q}(2)$ is isomorphic to the cobordism group ${\mathfrak {N}}_{2+q}(1,1)$, 
which is by definition the cobordism group of closed $(2+q)$-dimensional manifolds admitting one non-zero
vector field. If the closed manifold $N$ represents a class in ${\mathfrak {N}}_{2+q}(1,1)$, then 
 the tangent space $TN$ has one non-zero vector field and then $TN \oplus \ep^1_N$
 has two linearly independent vector fields: the second vector field is the natural framing of the bundle $\ep^1_N$ added to $TN$.
This is called stabilization and yields a representative of a class in $\CC_{2+q}(2)$. This is the natural stabilizing isomorphism
$${\mathrm {St}} \co {\mathfrak {N}}_{2+q}(1,1) \to \CC_{2+q}(2)$$ established in \cite{Ko}.

Then \cite[Theorem~12.8]{Ko} says that ${\mathfrak {N}}_{2+q}(1,1)$ is isomorphic to 
$\Z_2 \oplus {\mathfrak {N}}_{2+q}^{2|\chi}$ via a map %denoted by
$$(\chi'', f) \co {\mathfrak {N}}_{2+q}(1,1) \to  \Z_2 \oplus {\mathfrak {N}}_{2+q}^{2|\chi},$$ where
$f$ is just the forgetful map which forgets all the framings. The map $\chi''$ is defined as follows (see \cite[page 130 (12.5)]{Ko}).
Take a representative $N^{2+q}$ of a class $x \in {\mathfrak {N}}_{2+q}(1,1)$, then $TN = \ep^1_N \oplus \eta_N$
for some $(q+1)$-dimensional vector bundle $\eta_N$.
Let $S \subset N$ be the zero set of a generic smooth section of $\eta_N$. Then $S$ is a closed smooth $1$-dimensional manifold
%over $N - S$ the bundle $\eta$ has a new framing 
and
since $TS \oplus \eta_S = TN|_S = \ep^1_S \oplus \eta_S$ over $S$,
adding an appropriate bundle $\eta^{\perp}$ over $S$ such that $\eta_S \oplus \eta^{\perp}$ is a trivial bundle,  
we have a given stable parallelization of $TS$.
Hence $S$ represents an element in the bordism group $\Omega_1(point; \ep^1) \cong \Z_2$, see \cite[Definition~2.1 and the following explanation on page 27]{Ko}. This element is the value $\chi''(x)$. Notice that 
the isomorphism $\CC_{2+q}(2) \cong  \Z_2 \oplus {\mathfrak {N}}_{2+q}^{2|\chi}$ mentioned in the statement 
of the proposition
is just $(\chi'', f) \circ {\mathrm {St}}^{-1}$.

Now, we want to compute $$(\chi'', f) \circ {\mathrm {St}}^{-1} \circ \si_{2,q} ([g] )$$
of a class $[g \co M^{2+q} \to \R^2] \in \CC ob_{fr}(2,q)$.
The class  $[g \co M^{2+q} \to \R^2] \in \CC ob_{fr}(2,q)$ is mapped by $\si_{2,q}$ to the class of the stably $1$-framed source manifold $[M] \in \CC_{2+q}(2)$.
So we have $$TM \oplus \ep^1_M \cong \eta_M \oplus \ep^2_M$$ for some $\eta_M$.
Then we get ${\mathrm {St}}^{-1} ( [M] ) \in {\mathfrak {N}}_{2+q}(1,1)$ by the following process: 
since $[M]$ is in the image of the stabilizing isomorphism, 
the manifold  $M$ is stably $1$-framed cobordant by some $W$ to a manifold $M'$
such that 
\begin{enumerate}[\rm (1)]
\item
the bundle $TM' \oplus \ep^1_{M'}$ has two linearly independent vector fields,
\item
the second vector field of this framing coincides with the natural framing of the bundle $\ep^1_{M'}$.
\end{enumerate}
We can delete this second vector field together with the summand $\ep^1_{M'}$.
This means we obtain a representative $M'$ of the class ${\mathrm {St}}^{-1} ( [M] )$ such that $TM'$ has one non-zero vector field.
To get the value $$(\chi'', f) ( {\mathrm {St}}^{-1} ( [M] ) )$$ in the group 
$\Z_2 \oplus {\mathfrak {N}}_{2+q}^{2|\chi}$
at first observe that 
$f \circ  {\mathrm {St}}^{-1} ( [M] )$ is just the cobordism class $[M] \in {\mathfrak {N}}_{2+q}^{2|\chi}$.
 Now, we are going to look for $\chi'' \circ {\mathrm {St}}^{-1}([M]) = \chi'' ([M'])$.
 Computing $\chi'' ([M'])$ goes as above.
Namely, 
$TM' = \ep^1_{M'} \oplus \eta_{M'}$ for some bundle $\eta_{M'}$ and 
 we can find the stably parallelized $1$-dimensional manifold $S' \subset M'$
by taking  the zero set of a generic section of $\eta_{M'}$ over $M'$.
Of course we could find the same  $S' \subset M'$ 
in $TM' \oplus \ep^1_{M'} = \left( \ep^1_{M'} \oplus \eta_{M'} \right) \oplus \ep^1_{M'}$ instead of $TM'$ as well if we do not delete that second 
vector field specified in (2). 
Besides these,  we can find an $\eta_W$ in the entire $TW \oplus \ep^1_W$
which restricts to our $\eta_{M'}$ in $TM' \oplus \ep^1_{M'}$
and to $\eta_{M}$ in $TM \oplus \ep^1_{M}$.
And hence we find a stably parallelized $2$-dimensional manifold $\tilde S \subset W$ as the zero set of a generic section of $\eta_W$ in $TW \oplus \ep^1_W$.
Summarizing, we obtain that the zero set $S \subset M$ of a generic section of $\eta_M$ over the original manifold $M$, where 
we had $TM \oplus \ep^1_M \cong \eta_M \oplus \ep^2_M$,
is cobordant to $S'$ by the surface $\tilde S$ 
in the sense of the group $\Omega_1(point; \ep^1)$.

As a result we get that if we look for the value $\chi'' \circ {\mathrm {St}}^{-1} \circ \si_{2,q} \left([g \co M^{2+q} \to \R^2]\right)$,
then it is enough to consider the {\it stably} $1$-framed manifold $M$ 
as a representative in $\CC_{2+q}(2)$
and to find the zero set $S \subset M$ 
of a generic section of $\eta_M$ where $TM \oplus \ep^1_M \cong \ep^2_M \oplus \eta_M$, as we obtain from the map $g$.
Then this $S$ and its stable parallelization
give the same element in $\Omega_1(point; \ep^1)$ as $\chi'' \circ {\mathrm {St}}^{-1} \circ \si_{2,q} ([g])$. 

So having $TM \oplus \ep^1_M \cong \ep^2_M \oplus \eta_M$ obtained from a fold map
$g \co M^{2+q} \to \R^2$, we have to understand the  geometric meaning of this stable parallelized $S$.
We got the bundle $\eta_M$ as the {kernel} of the fiberwise epimorphism
 $$\tilde \va_{g, \varrho, r} \co TM \oplus \ep^1_M \to g^*T{\R^2}$$ obtained from the fold map $g$ by
applying
(\ref{fiberwiseepi}) in Section~\ref{existfrfold}. 
So the bundle $\eta_M$ is completely contained in $TM$ exactly at the singular points of $g$.
If we project fiberwisely the unit vector of $\ep^1_M$ to $\eta_M$ perpendicularly by the Riemannian metric used in Section~\ref{existfrfold}
(note that the direction $\ep^1_M$ is perpendicular to $TM$), then 
we get a continuous section $s$  of $\eta_M$, which is zero exactly at the singular set $S_g$ and smooth near $S_g$.
Approximate this continuous section $s$ by a smooth section of $\eta_M$, which coincides with $s$ near the singular set $S_g$.

Then the singular set $S_g$ is a smooth $1$-dimensional submanifold of $M$,
obviously $TS_g$ and $\eta_M|_{S_g}$ are subbundles of $TM|_{S_g}$ and we have
$$ \eta_M|_{S_g} \oplus TS_g =  TM|_{S_g},$$ 
\begin{equation}\label{isom}
\begin{array}{c}
%  \eta_M|_{S_g} \oplus TS_g =  TM|_{S_g},    \\ \\
\eta_M|_{S_g} \oplus TS_g  \oplus \ep^1_{S_g} = TM|_{S_g} \oplus \ep^1_{S_g} \cong \eta_M|_{S_g}  \oplus g^*T\R^2|_{S_g} 
\end{array}
% \eta_M|_{S_g} \oplus TS_g =  TM|_{S_g}, \\
%\eta_M|_{S_g} \oplus TS_g  \oplus \ep^1_{S_g} = TM|_{S_g} \oplus \ep^1_{S_g} \cong \eta_M|_{S_g}  \oplus g^*T\R^2|_{S_g} 
\end{equation} 
where the second isomorphism is being induced by $\tilde \va$. This equation
 gives the stable framing of $TS_g$ after adding $\eta_M^{\perp}|_{S_g}$ to both sides.
 Restricted to  $S_g$ the homomorphism $\tilde \va$ has the form 
\begin{equation}\label{va_form}
\eta_M|_{S_g} \oplus TS_g   \oplus \ep^1_{S_g} \to g^*T\R^2,
\end{equation}
$$
(u, v, w) \mapsto  dg(u, v) +  w,\mbox{\ \ \ \ \ \ \ \ \ \  }(u, v, w) \in \eta_M|_{S_g} \oplus TS_g \oplus \ep^1_{S_g}.
$$
 The isomorphism in (\ref{isom}) can be made explicit: if it is
 $$
 \iota \co \eta_M|_{S_g} \oplus TS_g  \oplus \ep^1_{S_g} \to \eta_M|_{S_g} \oplus g^*T\R^2|_{S_g},$$
 then 
 $$
 \iota ( u, v, w ) = (u, \tilde \va_{g, \varrho, r} ( 0, v, w ) ) =  (u, dg(0, v) +  w ).
 $$
 
%The bundle  $\eta_M|_{S_g}$ is the kernel of  the differential $dg$ as well at the points of $S_g$ as we can see from 
%(\ref{fiberwiseepi}).
Furthermore, since $dg|_{S_g}$ maps $\eta_M|_{S_g}$ to $0$ and $g|'_{S_g} \co TS_g \to g^*T\R^2$ coincides with
$$
v \mapsto dg (0, v), \mbox{\ \ \ \ \ }v \in TS_g,
$$
we have that the map $TS_g   \oplus \ep^1_{S_g} \to  g^*T\R^2$, $(v,w) \mapsto   dg (0, v) + w$ is the same as
the fiberwise isomorphism
$$g|'_{S_g} + \nu_g \co TS_g \oplus \ep^1_{S_g} \to g^*T\R^2$$ 
induced by 
 the  immersion $g|_{S_g} \co S_g \to \R^2$ and its trivial normal bundle $\nu_g$.

%,  we have the commutative diagram
%\begin{equation*}
%\begin{CD}
%TS_g   \oplus \ep^1_{S_g} @>   (v,w) \mapsto   dg (0, v) + w >>   g^*T\R^2   \\ %@>>>    g_* \circ dg|_{S_g} (0, v) + g_*|_{\ep^1_{S_g}} ( w )  \\
%@VVV @VV = V \\ %@VV = V \\
% \eta_M|_{S_g} \oplus TS_g   \oplus \ep^1_{S_g}  @>   (u,v,w) \mapsto dg (u, v) + w >>  g^*T\R^2, %@>>>     g_* \circ dg|_{S_g} (u, v) + g_*|_{\ep^1_{S_g}} ( w ),
%\end{CD}
%\end{equation*}
%for every 
%$(u, v, w) \in \eta_M|_{S_g} \oplus TS_g \oplus  \ep^1_{S_g}$,
%where the bottom row is the homomorphism $\tilde \va|_{S_g}$ and 
%upper row is just
%the fiberwise isomorphism
%$$g|'_{S_g} + \nu_g \co TS_g \oplus \ep^1_{S_g} \to g^*T\R^2$$ 
%induced by 
% the  immersion $g|_{S_g} \co S_g \to \R^2$ and its trivial normal bundle $\nu_g$.

Then  for some $k \in \N$ we have the  diagram
\begin{equation*}
\begin{CD}
\eta_M|_{S_g} \oplus TS_g   \oplus \ep^1_{S_g} @> \iota >>  \eta_M|_{S_g} \oplus g^* T\R^2|_{S_g}   \\
@V \eta^{\perp}_M|_{S_g} \oplus VV @V  \eta^{\perp}_M|_{S_g} \oplus V V  \\
\ep^k_{S_g} \oplus TS_g \oplus \ep^1_{S_g} @> \ka >> \ep^k_{S_g} \oplus g^* T\R^2|_{S_g} \\
@AAA @AAA  \\
TS_g \oplus \ep^1_{S_g} @> g|'_{S_g} + \nu_g  >> g^* T\R^2|_{S_g},   
\end{CD}
\end{equation*}
where the top vertical downward arrows are 
the identity isomorphisms on the direct summands $TS_g   \oplus \ep^1_{S_g}$ and $g^* T\R^2|_{S_g}$,
respectively. The homomorphism  $\ka$ is defined as
$\ka(u, v, w) = (u, dg(0, v) +  w)$ for $(u, v, w) \in \ep^k_{S_g} \oplus TS_g \oplus \ep^1_{S_g}$.
The bottom vertical upward arrows are inclusions.   These imply that the diagram is commutative.

Thus $g|'_{S_g} + \nu_g$ is stably equivalent to the isomorphism giving the stable framing of $TS_g$ through $\iota$
in (\ref{isom}).
Hence 
$$
\chi'' \circ {\mathrm {St}}^{-1} \circ \si_{2,q} \left([g \co M^{2+q} \to \R^2]\right) = [g|_{S_g}]
$$
in the group $\Z_2$.
\end{proof}

%Note that Proposition~\ref{stabfrcobsurj} together with computational results of \cite{Ko} yield some immediate 
%statements about existence of fold maps up to cobordism.
%For example, we have

%\begin{prop}
%Let $k \geq 2$ and $M^{4k}$ be a closed orientable manifold of dimension $4k$. 
%If $M$ has a fold map into $\R^4$, then $8 | \si(M)$.
%Conversely, if $8 | \si(M)$, then $M$ is cobordant to a closed orientable manifold having a fold map into $\R^4$.
%\end{prop}
%\begin{proof}
%This follows immediately from \cite[Theorem~19.41]{Ko}.
%\end{proof}

\section{The complete invariants of cobordisms of framed fold maps}\label{Complete_invariants_of_cobordisms_of_framed_fold_maps}

Recall that 
the group $\imm(\ep^1_X,n)$, where $X$ is a topological space,
is  identified with the group
$\pi^s_{n-1} \oplus \pi^s_{n-1}(X)$.
Throughout this section $\la$ denotes non-negative integers referring to the absolute indices of the fold singularities.

Denote by $\Im_{n,q}^O$ the homomorphism 
 $$\left( \si^{O}_{n,q} , \csi_{{1}}^O, \ldots, \csi_{\lfloor (q+1)/2 \rfloor}^O \right)$$
 and by $\Im_{n,q}$ the homomorphism
 $$\left(\si^{}_{n,q} , \csi_{{1}}, \ldots, \csi_{\lfloor (q+1)/2 \rfloor}\right).$$

\begin{thm}\label{invarithm}
Let $n \geq 1, q \geq 0$. %, $q + 1 \neq 4k$.
Then, the homomorphisms
\begin{multline*}
  \Im_{n,q} \co 
\CC ob_{fr}^{}(n,q) \longrightarrow    {\CC}_{n+q}{(n)}
\oplus  
\bigoplus_{1 \leq \la \leq (q+1)/2} \pi^s_{n-1} \oplus \pi^s_{n-1}\left(B(O(\la) \x O(q+1-\la))\right)
\end{multline*}
and
\begin{multline*}
  \Im_{n,q}^O
\co 
\CC ob_{fr}^{O}(n,q) \longrightarrow    {\CC}_{n+q}^O{(n)}
\oplus  
\bigoplus_{1 \leq \la \leq (q+1)/2} \pi^s_{n-1} \oplus \pi^s_{n-1}\left(BS(O(\la) \x O(q+1-\la))\right)
\end{multline*}
are injective. Hence two  framed fold maps $f_i \co Q_i^{n+q} \to \R^n$, $i=0,1$, are  framed  
cobordant if and only if 
$$\Im_{n,q}\left([f_0]\right) = \Im_{n,q}\left([f_1]\right)$$
and
two  oriented framed fold maps $f_i \co Q_i^{n+q} \to \R^n$, $i=0,1$, are oriented framed  
cobordant if and only if 
$$\Im_{n,q}^{O}\left([f_0]\right) = \Im_{n,q}^{O}\left([f_1]\right).$$
\end{thm}

The proof will be given in Section~\ref{completebiz}.
For $q$ even, since fold maps with even codimension can be framed naturally, in the statement of Theorem~\ref{invarithm}
the group
$\CC ob_{fr}^{O}(n,q)$ can be replaced by $\CC ob_{{}}^{O}(n,q)$ and
$\CC ob_{fr}^{}(n,q)$ can be replaced by $\CC ob_{{}}^{}(n,q)$.
For $q$ odd, there is a forgetful map $\CC ob_{fr}^{(O)}(n,q) \to \CC ob^{(O)}(n,q)$, which is 
obviously surjective if $n \leq 2$.

\begin{cor}\label{foldcobcor}
For $k \geq 0$,
the homomorphism $\Im_{n,2k}^{(O)}$ gives a complete invariant
of the (oriented) cobordism group $\CC ob_{{}}^{(O)}(n,2k)$ of fold maps.
\end{cor}

\begin{rem}\label{unoriequi}
By Corollary~\ref{foldcobcor} the homomorphisms
$$%\begin{multline*}
\Im_{n,0}^{O} \co 
\CC ob_{{}}^{O}(n,0) \to  {\CC}_{n}^O{(n)} 
{\mbox{\ \ \ and \ \ \ }}
\Im_{n,0}^{} \co 
\CC ob_{{}}^{}(n,0) \to  {\CC}_{n}{(n)}
$$%\end{multline*}
are injective, and by
Proposition~\ref{homsurj} they are surjective as well. Hence,
we have $$\CC ob_{{}}^{O}(n,0) \cong  {\CC}_{n}^O{(n)} {\mbox{\ \ \ and \ \ \ }} \CC ob_{{}}^{}(n,0) \cong  {\CC}_{n}{(n)}.$$
Since the group ${\CC}_{n}^O{(n)}$ is isomorphic to $\pi^s_n$, 
we obtain another argument for the isomorphism
$\CC ob_{{}}^{O}(n,0) \cong  \pi^s_n$ (for the original proof, see Ando \cite{An2, An}).
\end{rem}

\begin{cor}
\begin{enumerate}[\rm (1)]
\item
The cobordism group $\CC ob (2,0)$ of fold maps from unoriented surfaces into $\R^2$ is 
isomorphic to $\Z_2$. A fold map from an unoriented surface to $\R^2$ is null-cobordant if and only if 
its singular set is immersed into $\R^2$ with an even number of double points.
\item
The cobordism group $\CC ob (3,0)$ of fold maps from unoriented $3$-manifolds into $\R^3$ is 
isomorphic to $\Z_2$. 
A fold map from an unoriented $3$-manifold to $\R^3$ is null-cobordant if and only if 
the immersion of its  singular set into $\R^3$ is null-cobordant.
\end{enumerate}
\end{cor}
\begin{proof}
For (1) 
we have $\CC ob (2,0) \cong \CC_{2}(2) \cong  \Z_2 \oplus {\mathfrak {N}}_{2}^{2|\chi} \cong \Z_2$ 
because ${\mathfrak {N}}_{2}^{2|\chi} = 0$.
Proposition~\ref{koschizomexplained} gives the second part of the statement.
For (2) we have 
$\CC ob (3,0) \cong \CC_{3}(3) \cong  \Z_2 \oplus {\mathfrak {N}}_{3}^{2|\chi}$  by \cite[Theorem~12.8]{Ko} and
then $\CC ob (3,0) \cong \Z_2$
because ${\mathfrak {N}}_{3}^{2|\chi} = 0$.
Proposition~\ref{koschizomexplained} gives the second part of the statement.
\end{proof}

\begin{cor}
By \cite[Proposition~7.17 and Theorem~19.40]{Ko} we know that 
for $k \geq 1$ we have $\CC_{4k-1}^{O}(2) \cong \Omega_{4k-1}$.
The image of $\Im_{2,4k-3}^O$ is a subgroup of 
$\Omega_{4k-1} \oplus  (\Z_2 \oplus \Z_2)^{2k-1}$ and when $k=1$
this means that $\CC ob_{fr}^{O}(2,1) \subset \Z_2 \oplus \Z_2$.
It is easy to construct two framed fold maps $f_{1,2} \co M_{1,2} \to \R^2$ on $3$-manifolds such that
the index $1$ fold singular set of $f_1$ is immersed with one double point into $\R^2$, 
the index $1$ fold singular set of $f_2$ is immersed without double points into $\R^2$, 
the twisting $\tau_1 ( [f_1]) = 0$ and the twisting $\tau_1 ( [f_2]) = 1$, see, for example (3-1) and (1)-(4) in \cite[page~328]{Kal6} with the role $N^n = \R^2$.
Then the cobordism invariants $t_1$ and $\tau_1$ distinguish between $[f_1]$ and $[f_2]$,
hence $\CC ob_{fr}^{O}(2,1) \cong \Z_2 \oplus \Z_2$.
The forgetful map $\CC ob_{fr}^{O}(2,1)  \to \CC ob^{O}(2,1)$ is clearly surjective and even in 
$ \CC ob^{O}(2,1)$ the invariants $t_1$ and $\tau_1$ distinguish between $[f_1]$ and $[f_2]$, so
$ \CC ob^{O}(2,1) \cong \Z_2 \oplus \Z_2$. For other proofs of this fact not using h-principle, see \cite{Kal2} and \cite[Theorem~2.9]{Kal6}.
\end{cor}

\section{Morse function bundles over immersions}\label{mobuimm}

In this section, we recall some results of \cite{Kal7} for the convenience of the reader.
These results will be used in Sections~\ref{proof2}, \ref{proof3} and \ref{proof4}.

\subsection{Cobordism classes of Morse function bundles}

For $q \geq 2$ and $1 \leq j < (q+1)/2$
 we construct fold maps $\varphi_{j, q}$ of
some $(2+q)$-dimensional manifolds into $\R^2$, where the $\varphi_{j, q}$ will also depend on some other parameters. 
The cobordism classes of these fold maps $\varphi_{j, q}$ will serve
as generators of an important direct summand of the cobordism group $\CC ob_{{}}^O(2,q)$.
We will construct similar maps in the unoriented case as well.

For $q \geq 2$ and $1 \leq j < (q+1)/2$, 
let $h_j \co S^{q+1} \to \R$ be a Morse function of the $(q+1)$-dimensional sphere onto the
closed interval $[-\ep,\ep]$
with four critical points $a, b, c, d \in S^{q+1}$ of index 
$0, j-1, j, q+1$, respectively, such that $h_j(a) = -\ep$, $h_j(b) = -\ep/2$, $h_j(c) = 0$ and
$h_j(d) = \ep$.
Recall the following result from \cite{Kal7}.
\begin{lem}[Lemma~3.2 \cite{Kal7}]\label{kiterj}
There exists an identification of the 
Morse function $h_j$ around its critical point of index $j$
with the fold germ
\[
\ga(x_1,\ldots,x_{q+1})=\left(-x_1^2 - \cdots -x_{j}^2 + x_{1+j}^2 + \cdots + x_{1+q}^2\right),
\]
such that under this identification
\begin{enumerate}[\textup{(}\rm 1\textup{)}]
\item
any automorphism in the automorphism group $O(1) \x O(q)$ \textup{(}in the case of $j=1$\textup{)}
\item
any automorphism in the automorphism group\footnote{Let $O(\mathbf{1},k)$ denote the subgroup of the orthogonal group $O(k+1)$ 
whose elements are of the form $\begin{pmatrix}
1 & 0 \\ 0 & M 
\end{pmatrix}$
where $M$ is an element of the group $O(k)$.} 
$O(\mathbf{1},j-1) \x O(q+1-j)$ 
%$O(q+1-j)$ 
\textup{(}in the case of $j>1$\textup{)} 
\end{enumerate}
of the fold germ $\ga$
can be extended to an automorphism of the Morse function $h_j$\textup{.}
\end{lem}

Following \cite[Section~3]{Kal7} in the special case of $n = 2$, 
we define the group homomorphisms
$$\al_{1} \co \imm\left({\ep^1_{BS(O(1) \x O(q))}},2\right) \to \CC ob_{{}}^O(2,q)$$
and $$\al_{j} \co \imm\left({\ep^1_{BS(O(\mathbf{1},j-1) \x O(q+1-j))}},2\right) \to \CC ob_{{}}^O(2,q)$$
 for $2 \leq j < (q+1)/2$ as follows.

We first define $\al_{1}$. 
Let $[m \co M^{1} \to \R^{2}]$ be an element of $\imm\left({\ep^1_{BS(O(1) \x O(q))}},2\right)$.
Then the normal bundle of the immersion $m$ is 
induced from the trivial line bundle $\ep^1_{BS(O(1) \x O(q))}$ by a map 
$$\mu \co M^{1} \to BS(O(1) \x O(q)).$$
By Lemma~\ref{kiterj} the symmetries in $S(O(1) \x O(q))$ extend to symmetries of the Morse function $h_1$.
Hence the inducing map $\mu$ yields a bundle with fiber 
 the Morse function $h_1$ and base space the $1$-dimensional manifold $M^{1}$.
 The Morse function $h_1$ has the source $S^{q+1}$ and the target $[-\ep, \ep]$ so
 this $h_1$-bundle over $M^1$ consists of an $S^{q+1}$ bundle over $M^1$, an $[-\ep, \ep]$ bundle over $M^1$ and
 a map $\be$ between the total spaces of these two latter bundles. The map $\beta$ fiberwise can be identified with $h_1$.
 We denote this $S^{q+1}$ bundle over $M^1$ by $p_s \co Q_{1,q,m}^{2+q} \to M^1$,
 this $[-\ep, \ep]$ bundle over $M^1$ by $p_t \co J \to M^1$ and then
 $\be$ maps $Q_{1,q,m}^{2+q}$ to $J$ (and as we said the restriction of $\be$ to a fiber $S^{q+1}$ can be identified with $h_1 \co S^{q+1} \to [-\ep, \ep]$). So we have the commutative diagram

\begin{center}
\begin{graph}(6,2)
\graphlinecolour{1}\grapharrowtype{2}
\textnode {A}(0.5,1.5){$Q_{1,q,m}^{2+q}$}
\textnode {B}(5.5, 1.5){$J$}
\textnode {C}(3, 0){$M^1$}
\diredge {A}{B}[\graphlinecolour{0}]
\diredge {B}{C}[\graphlinecolour{0}]
\diredge {A}{C}[\graphlinecolour{0}]
\freetext (3,1.8){$\beta$}
\freetext (1.2, 0.6){$p_s$}
\freetext (4.8, 0.6){$p_t$}
\end{graph}
\end{center}
where $\be$ maps fiberwise.
 
Obviously, since $M^1$ is immersed into $\R^2$ (by $m$), 
the total space of the normal bundle of $m$ is also immersed into $\R^2$ (by $\nu$, say).
Identify $J$ with the normal bundle of $m$ (in fact 
this normal bundle is a line bundle having an $[-\ep, \ep]$ subbundle and this and $J$
 are the same $[-\ep, \ep]$ bundle over $M^1$, i.e.\ both of them are induced by $\mu$) and
 compose $\be$ with $\nu$ to get
the fold map $\nu \circ \be$, which we denote by 
$$\va_{1,q, m} \co  Q_{1,q,m}^{2+q} \to \R^2.$$
Now
let $\al_{1}([m])$ be 
 the fold cobordism class of  $\va_{1,q, m}$. This definition depends only on the immersion cobordism class of $m$ as one can see easily by doing the analogous constructions for cobordisms.

The definition of the group homomorphism $\al_{j}$ is similar:
for an element $[m']$ in $\imm\left({\ep^1_{BS(O(\mathbf{1},j-1) \x O(q+1-j))}},2\right)$
 we define
the fold map 
 $\va_{j,q,m'} \co Q_{j,q,m'}^{2+q} \to \R^2$ and its
  cobordism class $\al_{j}([m'])$ for $j>1$ in a completely analogous way.
  
  Now for formal reasons let $O(\mathbf{1}, 0)$ denote just the group $O(1)$, which is not a very good notation but we will be using it.
    For convenience, we extend each $\al_j$ to the other $i \neq j$ summands of the group
$$\bigoplus_{1\leq i < (q+1)/2} \imm\left({\ep^1_{BS(O(\mathbf{1}, i-1) \x O(q+1-i))}},2\right)$$ as the identically zero homomorphism.

\subsection{Direct summands of fold cobordism groups}
    
 We have the following statement (see also    \cite[Remark~3.3]{Kal7}).
\begin{prop}\label{directsumincobgroup}
The homomorphism 
 $\sum_{1\leq j < (q+1)/2} \al_{j}$
 is an isomorphism onto a direct summand of $\CC ob_{{}}^O(2,q)$.
 The group $\CC ob_{{}}^O(2,q)$ contains $\Z_2^{2\lfloor \frac{q}{2} \rfloor}$ as a direct summand.
\end{prop}
\begin{proof}
At first, notice that for all $2 \leq j < (q+1)/2$, since we consider immersions into the plane, the group 
$$\imm\left({\ep^1_{BS(O(\mathbf{1},j-1) \x O(q+1-j))}},2\right)$$ is the same as the group $$\imm\left({\ep^1_{BS(O(j) \x O(q+1-j))}},2\right),$$
so for all $j$ the homomorphisms $\al_j$ are in fact homomorphisms of type
$$\al_j \co \imm\left({\ep^1_{BS(O(j) \x O(q+1-j))}},2\right) \to \CC ob_{{}}^O(2,q).$$
As before, for convenience, we extend each $\al_j$ to the other $i \neq j$ summands of the group
$$\bigoplus_{1\leq i < (q+1)/2} \imm\left({\ep^1_{BS(O(i) \x O(q+1-i))}},2\right)$$ as the identically zero homomorphism.

Take the composition 
\begin{multline*}
\bigoplus_{1\leq j < (q+1)/2} \imm\left({\ep^1_{BS(O(j) \x O(q+1-j))}},2\right) \longrightarrow
\CC ob_{{}}^O(2,q) \longrightarrow  \\ \bigoplus_{1\leq j < (q+1)/2} \imm\left({\ep^1_{BS(O(j) \x O(q+1-j))}},2\right),
\end{multline*}
where the first arrow is the homomorphism $\sum_{1\leq j < (q+1)/2} \al_{j}$
and the second arrow is the homomorphism $(\csi^O_{1}, \ldots,  \csi^O_{\lfloor q/2 \rfloor})$.

If we show that this composition
$$(\csi^O_{1}, \ldots,  \csi^O_{\lfloor q/2 \rfloor})  \circ \sum_{1\leq j < (q+1)/2} \al_{j}$$
is an isomorphism, then the statement of the proposition follows.

For each $1 \leq j < (q+1)/2$, we fix a basis $\{ [i_j], [e_j] \}$ of the domain of $\al_j$, which is
\begin{multline*}
\imm\left({\ep^1_{BS(O(j) \x O(q+1-j))}},2\right) 
\cong 
%\pi^s_2(T\ep^1_{BS(O(j) \x O(q+1-j))}) \cong  \\ 
\pi_{1}^s \oplus
\pi_{1}^s\left( BS( O(j) \x O(q + 1-j))\right)
\cong \Z_2 \oplus \Z_2.
\end{multline*}
 Let $i_j \co S^1 \to \R^2$ be an immersion with $1$ double point and 
with trivial normal bundle induced from the bundle
$\ep^1_{BS(O(j) \x O(q+1-j))}$
by a constant map $S^1 \to {BS(O(j) \x O(q+1-j))}$.
Hence $[i_j]$ represents $(1, 0)$ in this $\Z_2 \oplus \Z_2$.
Let $e_j \co S^1 \to \R^2$ be an immersion without any multiple points  and 
with the normal bundle induced 
from the bundle $\ep^1_{BS(O(j) \x O(q+1-j))}$
by a map $S^1 \to {BS(O(j) \x O(q+1-j))}$
which represents the non-trivial element in 
\begin{multline*}
\Z_2 \cong 
\pi_{1}^s\left( BS( O(j) \x O(q + 1-j))\right) = \pi_{N+1}\left({S^N}BS( O(j) \x O(q + 1-j))\right)
\cong \\ H_{N+1}\left(S^N BS(O(j) \x O(q+1-j)); \Z\right) = H_{1}\left(BS(O(j) \x O(q+1-j)); \Z\right),
\end{multline*}
where ``$S^N$'' denotes the $N$th suspension for a large $N$.
Then the index $j$ indefinite fold singular set of $\va_{j,q,e_j}$ is a circle whose  normal bundle in $Q_{j,q,e_j}^{2+q}$ has the gluing transformation which
 is orientation reversing on both of the $O(j)$-  and $O(q+1-j)$-invariant subspaces, as one can see easily.
Then $[e_j]$ represents $(0,1)$ in $\Z_2 \oplus \Z_2$.

For a $1 \leq j' < (q+1)/2$, what are  $\csi^O_{j'} \circ \al_{j} ([i_j])$ and $\csi^O_{j'} \circ \al_{j} ([e_j])$?
We know that $\al_{j} ([i_j]) = [\va_{j,q,i_j}]$ and $\al_{j} ([e_j]) = [\va_{j,q,e_j}]$.
The fold maps $\va_{j,q,i_j}$ and $\va_{j,q,e_j}$ have fold singularities of absolute indices $0$, $j-1$ and $j$.
Besides the fact that $\csi^O_{j} ([\va_{j,q,i_j}]) = [i_j]$ and $\csi^O_{j} ([\va_{j,q,e_j}]) = [e_j]$
always hold we have that $\csi^O_{j'} ([\va_{j,q,i_j}]) $ or $\csi^O_{j'}([\va_{j,q,e_j}]) $ can be non-zero only if $j' = j-1$.
This shows that the matrix of 
 $$(\csi^O_{1}, \ldots,  \csi^O_{\lfloor q/2 \rfloor})  \circ \sum_{1\leq j < (q+1)/2} \al_{j}$$
is an upper triangular matrix with $1$s along the diagonal. This finishes the proof.
\end{proof}

\begin{rem}\label{symm}
A little more information is that we have 
$$\csi^O_{1} ([\va_{2,q,i_2}])= [i_1], \csi^O_{2} ([\va_{3,q,i_3}])= [i_2], \ldots, 
 \csi^O_{\lfloor q/2 \rfloor -1} ([\va_{\lfloor q/2 \rfloor,q,i_{\lfloor q/2 \rfloor}}]) = [i_{\lfloor q/2 \rfloor-1}].$$
 
 To have the similar result for the $[e_j]$s we have to see  how
 the symmetry of the Morse function $h_j$ which acts non-trivially on the critical point of index $j$
 acts on the critical point of index $j-1$ for $j-1 \geq 1$.
 By \cite[Proof of Lemma~3.2]{Kal7}) we have that for $2 \leq j \leq \lfloor q/2 \rfloor$,
 $$\csi^O_{j-1} ([\va_{j,q,e_{j}}])= [e_{j-1}].$$
\end{rem}

For the unoriented case let 
$$\tilde \al_j \co  \imm\left({\ep^1_{B(O(j) \x O(q+1-j))}},2\right) \to \CC ob_{{}}(2,q),$$ 
$1 \leq j < (q+1)/2$, be the homomorphisms like the $\al_j$.
Note that
\begin{multline*}
\imm\left({\ep^1_{B(O(j) \x O(q+1-j))}},2\right) 
\cong 
\pi_{1}^s \oplus
\pi_{1}^s\left( B( O(j) \x O(q + 1-j))\right)
\cong \Z_2 \oplus \Z_2 \oplus \Z_2.
\end{multline*}

Similarly to the previous arguments, we obtain

\begin{prop}\label{directsumincobgroupunori}
The homomorphism 
 $\sum_{1\leq j < (q+1)/2} \tilde \al_{j}$
 is an isomorphism onto a direct summand of $\CC ob(2,q)$.
 The group $\CC ob_{{}}(2,q)$ contains $\Z_2^{3\lfloor \frac{q}{2} \rfloor}$ as a direct summand.
\end{prop}
\begin{proof}
The proof is completely analogous to the proof of Proposition~\ref{directsumincobgroup}.

For each $1 \leq j < (q+1)/2$, we take the standard basis $\{ [i_j], [e_j^1], [e_j^2] \}$ of the domain $\Z_2^3$ of $\tilde \al_j$.
The immersions $i_j, e_j^1$ and $e_j^2$
are defined as follows.
Each of them maps $S^1$ into $\R^2$, the immersion $i_j$ has one double point, the immersions 
$e_j^1$ and $e_j^2$ have no multiple points.
The normal bundle of each of them 
 is induced from the bundle $\ep^1_{B(O(j) \x O(q+1-j))}$.
 The normal bundle of $i_j$ is induced by the constant map $S^1 \to B(O(j) \x O(q+1-j))$,
 the normal bundle of $e_j^1$ by a map $S^1 \to B(O(j) \x O(q+1-j))$
 which represents the element $(1, 0)$ in $\pi_{1}^s\left( B( O(j) \x O(q + 1-j))\right) \cong \Z_2 \oplus \Z_2$, i.e.\
 twists the $O(j)$ component but not the $O(q + 1-j)$ component,
  and
 the normal bundle of $e_j^2$ by a map $S^1 \to B(O(j) \x O(q+1-j))$
 which represents the element $(0, 1)$ in $\pi_{1}^s\left( B( O(j) \x O(q + 1-j))\right)$ so
 it does not twist the $O(j)$ component but twists the $O(q + 1-j)$ component.

Then we construct the fold maps $\tilde \va_{j, q, m} \co \tilde Q_{j, q,  m}^{2+q} \to  \R^2$ just like in the oriented case, where 
$m$ runs over the elements of $\{ i_j, e_j^1, e_j^2 : 1 \leq j < (q+1)/2 \}$.

Then we show that the matrix of the homomorphism
$$(\csi_{1}, \ldots,  \csi_{\lfloor q/2 \rfloor})  \circ \sum_{1\leq j < (q+1)/2} \tilde \al_{j}$$
is non-singular (over the field $\Z_2$).
Details are left to the reader.
\end{proof}

\section{Fold maps into the plane and a Poincar\'e-Hopf type formula for the signature}\label{foldmapsplanePoinHopf}

\subsection{Results about oriented fold maps}
Let $n=2$, $q \geq 0$  and 
let $0 \leq \la \leq (q+1)/2$.
Recall from Section~\ref{geomcobinv}
that the group
$$ \pi_{1}^s \oplus \pi_{1}^s\left( BS( O(\la) \x O(q + 1-\la))\right)$$
for $0 \leq \la < (q+1)/2$ and the group
$$\imm({\tilde l^1},2)$$ for the case of $q$ odd and $\la = (q+1)/2$,
which are the targets of the homomorphisms $\csi_{\la}^O$, are isomorphic to $\Z_2 \oplus \Z_2$ if $\la \geq 1$ and  to
 $\Z_2 \oplus \{ 0 \}$ if $\la = 0$.
Recall that the first component of the homomorphism $\csi_{\la}^O$, where $0 \leq \la \leq (q+1)/2$,
i.e.\ the homomorphism 
$$t_{\la} \co  \CC ob_{{}}^O(2, q) \to \Z_2$$
maps a fold cobordism class $[f]$ to the cobordism class of the immersion 
$$f|_{S_{\la}(f)} \co S_{\la}(f) \to \R^2$$
of the $1$-dimensional 
manifold $S_{\la}(f)$.
The cobordism class of this immersion 
  is an element\footnote{If the immersion $f|_{S_{\la}(f)}$ is in general 
position, then its cobordism class is equal to the number of its double points modulo $2$. Since
our fold maps are in general position, $f|_{S_{\la}(f)}$ is also in general position.} in $\Z_2$.
Simplifying the notation we often refer to $t_{\la}([f])$ as $t_{\la}(f)$. It is just the  mod $2$ number of double points 
of the $f$-image in $\R^2$ of the  index $\la$ fold singular set  of the generic fold map $f$. 

The second component of the homomorphism $\csi_{\la}^O$ for $0 \leq \la \leq (q+1)/2$ is
$$\tau_{\la} \co  \CC ob_{{}}^O(2, q) \to \Z_2,$$
which
maps a fold cobordism class $[f]$ to the sum $\tau_{\la}([f]) = \sum_{r} \tau_{\la,r}([f])$ mod $2$, where 
$\tau_{\la,r}([f])$ is $0$ if 
\begin{itemize}
\item
the
$S(O(\la) \x O(q+1-\la))$ bundle is trivial 
in the case of $0 \leq \lambda < (q+1)/2$ and
\item
the $S{\left\langle O\left(\frac{q+1}{2}\right) \x O\left(\frac{q+1}{2}\right), T \right\rangle}$ bundle is trivial
in the case of $q$ odd and $\lambda = (q+1)/2$
\end{itemize}
 over the $r$th 
component of the $1$-dimensional manifold
$S_{\la}(f)$, and $1$ otherwise\footnote{This implies that $\tau_0 \co  \CC ob^O(2, q) \to \Z_2$ is always the  zero homomorphism.}, 
cf.\ Remark~\ref{nfrori}.

\begin{rem}\label{detnyaltwist}
In other words, when $q$ is even,
$\tau_{\la}([f])$ is equal to the first Stiefel-Whitney number 
$\langle w_1(\delta_{\la}(f)), [S_{\la}(f)]\rangle$
of the determinant bundle $\delta_{\la}(f)$ of the $O(\la)$
bundle obtained by the projection $S(O(\la) \x O(q+1-\la)) \to O(\la)$ over $S_{\la}(f)$.
\end{rem}
We refer to $\tau_{\la}([f])$ as $\tau_{\la}(f)$ and say it is the {\it twisting} of the index $\la$ fold germs over $S_{\la}(f)$.
Now, take  the homomorphisms
$$t \co \CC ob_{{}}^O(2,q) \to \Z_2,$$
$$t = \sum_{0 \leq \la \leq (q+1)/2} t_{\la}$$
 and
$$\tau \co \CC ob_{{}}^O(2,q) \to \Z_2,$$
$$\tau = \sum_{0 \leq \la \leq (q+1)/2} \tau_{\la}.$$

Some results of \cite{Chess} can be reformulated as follows.

\begin{thm}[Chess \cite{Chess}]
Let $f \co M^{2k+1} \to \R^2$ be a fold map of a closed orientable manifold.
Then
\[
 t(f) + \tau(f) \equiv \left \{
 \begin{array}{ll}
 0 \mod{2} & \mbox{if $k$ is odd,} \\
 w_2 w_{2k-1}[M^{2k+1}]  \mod{2} & \mbox{if $k$ is even.}
 \end{array}
 \right. \]
\end{thm}

We are looking for a similar result if the dimension of the  source manifold of a fold map is divisible by $4$.
To achieve this, we will compute some related cobordism groups of fold maps.

Framed fold maps with even codimension are naturally identified with fold maps.
Hence when $n = 2$, $q \geq 0$ and $q = 2q'$ is even, 
Theorem~\ref{invarithm} says that $\Im_{2,q}^O$ is an injective   homomorphism   from the group $\CC ob_{{}}^{O}(2,q)$ 
 to the group
 \begin{multline*}
 {\CC}_{2+q}^O{(2)} \oplus 
 \pi^s_{1} \oplus \pi^s_{1}\left(BS\left(O\left(1\right) \x O\left(q\right)\right)\right) \oplus \cdots \oplus \pi^s_{1} \oplus \pi^s_{1}\left(BS\left(O\left(\frac{q}{2}\right) \x O\left(\frac{q}{2}+1\right)\right)\right).
 \end{multline*}
This large direct sum is actually isomorphic to ${\CC}_{2+q}^O{(2)}  \oplus \Z_2^{q}$.
Then we have the following.
\begin{thm}\label{explsikbathm}
For $k \geq 1$, the group  $\CC ob_{{}}^O(2, 4k-2)$
is isomorphic to $\Omega_{4k}^{2|\chi} \oplus \Z_2^{4k-2}$.
An isomorphism is given by the map
$$[f \co M^{4k} \to \R^2] \mapsto \left( [M^{4k}],  t_1(f), \tau_{1}(f), \ldots, t_{2k-1}(f), \tau_{2k-1}(f) \right).$$
\end{thm}

For example, the group $\CC ob_{{}}^O(2,2)$ is isomorphic to 
$\Z \oplus \Z_2 \oplus \Z_2$, and an isomorphism is given by the homomorphism
$$[f \co M^{4} \to \R^2] \mapsto \left( \frac{\si(M^4)}{2} , [f|_{S_1(f)}] , \langle w_1(\delta_1(f)), [S_1(f)]\rangle\right).$$

The next proposition will be important for us. We will prove later.

\begin{prop}\label{prop2CP^2-nfoldmap}
For $k \geq 1$ there is an oriented fold map $f \co \CP^{2k} \# \CP^{2k} \to \R^2$
such that $t(f)  \equiv 0 \mod{2}$ and
$\tau(f) \equiv 1 \mod{2}$.
\end{prop}
\begin{proof}
This is the same statement as Proposition~\ref{prop2CP^2-nfoldmapproof-al}, see the proof there.
\end{proof}

%If the dimension of the oriented source manifold of a fold map is divisible by $4$, then 
%we have the following result.
We obtain the following Poincar\'e-Hopf type formula for the signature.

\begin{thm}\label{sign}
Let $k \geq 1$ and $f \co M^{4k} \to \R^2$
be a fold map of a closed oriented $4k$-dimensional manifold. Then
$$\frac{\si( M^{4k} ) }{2} \equiv t(f) + \tau(f)  \mod{2}.$$
%This congruence also holds for a fold map
% $f \co M^{4k} \to S^2$  of
% a closed oriented $4k$-dimensional manifold into the $2$-sphere if there exists a regular value
% $y \in S^2$ such that the fiber $f^{-1}(y)$ is an oriented null-cobordant $(4k-2)$-dimensional manifold.
 \end{thm}

%\begin{rem}
%This congruence also holds for a fold map
% $f \co M^{4k} \to S^2$  of
% a closed oriented $4k$-dimensional manifold into the $2$-sphere if there exists a regular value
% $y \in S^2$ such that the fiber $f^{-1}(y)$ is an oriented null-cobordant $(4k-2)$-dimensional manifold.
%\end{rem}

\begin{cor}\label{4dimben}\noindent
An isomorphism $\CC ob_{{}}^O(2,2) \to \Z \oplus \Z_2^2$ is also given by the map
$$[f \co M^{4} \to \R^2] \mapsto \left( \si(M^{4})/2 , 
[f|_{S_0(f)}] , [f|_{S_1(f)}] \right).$$
\end{cor}
\begin{proof}
This follows from Theorems~\ref{explsikbathm} and \ref{sign}.
\end{proof}

\subsection{Results about non-oriented fold maps}

Now, for the unoriented case take  the homomorphisms
$$t_{\la} \co \CC ob_{{}}(2,4k-2) \to \Z_2$$ and
$$\tau_{\la} \co \CC ob_{{}}(2,4k-2) \to \Z_2 \oplus \Z_2,$$
where 
$0 \leq \la \leq 2k-1$.
Define $\tilde \tau_{\la}$ as the product  $\tau^1_{\la} \tau^2_{\la}$ taken in the field $\Z_2$ of the two components of 
$\tau_{\la}$ and
define
$$\tilde \tau = \sum_{0 \leq \la \leq 2k-1} \tilde \tau_{\la}.$$

It follows easily that  for $0 \leq \la \leq 2k-1$ we have the commutative diagrams

\begin{center}
\begin{graph}(6,2)
\graphlinecolour{1}\grapharrowtype{2}
\textnode {A}(0.5,1.5){$\CC ob_{{}}^O(2,4k-2)$}
\textnode {B}(5.5, 1.5){$\CC ob_{{}}(2,4k-2)$}
\textnode {C}(3, 0){$\Z_2$}
\diredge {A}{B}[\graphlinecolour{0}]
\diredge {B}{C}[\graphlinecolour{0}]
\diredge {A}{C}[\graphlinecolour{0}]
\freetext (3,1.8){$\iota$}
\freetext (1.1, 0.6){$t_{\la}$}
\freetext (4.9, 0.6){$t_{\la}$}
\end{graph}
\end{center}
and
\begin{center}
\begin{graph}(6,2)
\graphlinecolour{1}\grapharrowtype{2}
\textnode {A}(0.5,1.5){$\CC ob_{{}}^O(2,4k-2)$}
\textnode {B}(5.5, 1.5){$\CC ob_{{}}(2,4k-2)$}
\textnode {C}(3, 0){$\Z_2$}
\diredge {A}{B}[\graphlinecolour{0}]
\diredge {B}{C}[\graphlinecolour{0}]
\diredge {A}{C}[\graphlinecolour{0}]
\freetext (3,1.8){$\iota$}
\freetext (1.1, 0.6){$\tau_{\la}$}
\freetext (4.9, 0.6){$\tilde \tau_{\la}$}
\end{graph}
\end{center}
where the horizontal arrows denoted by $\iota$ are the natural forgetful homomorphisms.
Of course $\tilde \tau_{\la}$ is not necessarily a homomorphism. But the commutative diagram above
says that $\tilde \tau_{\la}$ restricted to the image $\iota \left( \CC ob_{{}}^O(2,4k-2) \right)$
is a homomorphism.

Notice that for $0 \leq \la \leq 2k-1$ the composition
\begin{equation*}
\begin{CD}
\CC ob_{{}}^O(2,4k-2) @>>> \CC ob(2,4k-2) @> \tau_{\la}^1 + \tau_{\la}^2 >>  \Z_2   \\
\end{CD}
\end{equation*}
is identically zero. Here the first arrow is the forgetful homomorphism and 
$\tau_{\la}^i$, $i = 1, 2$, denote the two components of $\tau_{\la}$.
We obtain  results in the unoriented case analogously to Theorems~\ref{explsikbathm} and \ref{sign}.

\begin{thm}\label{explsikbathmunori}
For $k \geq 1$, the cobordism group  $\CC ob_{{}}(2, 4k-2)$
is isomorphic to ${\mathfrak N}_{4k}^{2|\chi}  \oplus  \Z_2  \oplus \Z_2^{6k-3}$.
%An isomorphism is given by the map
%\begin{multline*}
%[f \co M^{2k} \to \R^2] \longmapsto \\ \left( [M^{2k}] ,  t_0(f) ,  t_1(f) , \tau_{1}^1(f) , \tau_{1}^2(f),  t_2(f) , \tau_{2}^1(f) , \tau_{2}^2(f),
%%[f|_{S_{2}(f)}] , \tau_{2}^1(f) , \tau_{2}^2(f), 
%\ldots,  t_{k-1}(f) , \tau_{k-1}^1(f) , \tau_{k-1}^2(f) \right).
%\end{multline*}
\end{thm}

\begin{thm}\label{unorirel}
Let $k \geq 1$ and $f \co M^{4k} \to \R^2$
be a fold map of a closed  (possibly unorientable) $4k$-dimensional manifold. Then
$$\tau^1(f) \equiv \tau^2(f)   \mod{2}.$$
%This congruence also holds for a stable map
% $f \co M^{4k} \to S^2$  of
% a closed oriented $4k$-dimensional manifold into the $2$-sphere if there exists a regular value
% $y \in S^2$ such that the fiber $f^{-1}(y)$ is an oriented null-cobordant $(4k-2)$-dimensional manifold.
 \end{thm}

This statement involves no number of double points and strictly speaking no topological properties of the manifold $M$, it relates only the twistings 
$\tau^1$ and $\tau^2$ to each other.

%\begin{ques}
%By \cite{Lev1}, for an $f \co M^{4k} \to \R^2$
% stable map of a closed oriented $4k$-dimensional manifold the congruences $\si( M^{4k} ) \equiv \chi( M^{4k} ) \equiv c(f)  \mod 2$
%hold,
%and by Proposition~\ref{sign} the parity of the integer  $(\si( M^{4k} ) - c(f))/2$ measures
%how complicated is the restriction of 
%$f$ to a small tubular neighbourhood of its singular set.
%What does the integer ${}rac{{\si( M^{4k} ) } - c(f) - 2(t(f) + \del([f]))}{4}$ measure for 
% a stable map $f \co M^{4k} \to \R^2$ of a closed oriented $4k$-dimensional manifold?
%\end{ques}

\section{Proof of the results}\label{completebiz}

\subsection{Proof for the cobordism invariants}\label{proof1}

At first, we prove Theorem~\ref{invarithm}.

\begin{proof}[Proof of Theorem~\ref{invarithm}]
Let $f \co Q^{n+q} \to \R^n$ be a framed fold map. We show that if 
$$\Im_{n,q}([f]) = \left(\si_{n,q}([f]) , \csi_{{1}}([f]), \ldots, \csi_{\lfloor (q+1)/2 \rfloor}([f])\right)$$
is zero, then $[f] \in \CC ob_{fr}(n,q)$ is also zero.
Take the map 
$$f \x {\mathrm {id}}_{[0,3\ep)} \co Q \x [0,3\ep) \to \R^n \x [0,3\ep)$$
for some small $\ep >0$.

Recall that 
$$\va \co \R^{q+1} \to \R,$$
$$\va(x_1,\ldots,x_{q+1})  = \sum_{i = 1}^{\la}
-x_i^2  + \sum_{i = \la+1}^{q+1} x_{i}^2$$
is the fold singularity of index $\la$.
By assumption
 \[\csi_{\la}([f]) \in \imm\left({\ep^1_{B(O(\la) \x O(q+1-\la))}},n\right)\] 
is zero for $1 \leq \la \leq (q+1)/2$ so
we can glue these given null-cobordisms of the  
$\va$-bundles of $f$ to the map
$$f \x {\mathrm {id}}_{[0,3\ep)}$$
``at''  $Q \x \{ 3\ep \}$.
Such a null-cobordism is a $\va$-bundle over an $n$-dimensional manifold $\Si_{\la}$ whose 
boundary $\del \Si_{\la}$ is the fold singular set $S_{\la}(f)$. 
In other words, there is an $\R^{q+1}$ bundle over $\Si_{\la}$ with the total space $\mathcal O_{\la}$,  and an $\R$ bundle over $\Si_{\la}$
with the total space $\mathcal R_{\la}$
 and there is a map $\Phi_{\la} \co \mathcal O_{\la} \to \mathcal R_{\la}$
which maps fiberwise as $\va$.
Also $\Si_{\la}$ is immersed into $\R^n \x [0,1)$  so that 
this immersion restricted to the boundary $S_{\la}(f)$ is the immersion $f|_{S_{\la}(f)}$ into $\R^n \x \{ 0 \}$.
So we have a commutative diagram
\begin{equation*}
\begin{CD}
D\mathcal N (S_{\la}(f)) @> \subset >> D\mathcal O_{\la} @> \Phi_{\la}|_{D\mathcal O_{\la}} >>  D\mathcal R_{\la}  \\
@V \subset VV @V \subset VV @V \subset VV  \\
\mathcal N (S_{\la}(f)) @> \subset > i > \mathcal O_{\la} @> \Phi_{\la} >>  \mathcal R_{\la}   \\
@VVV @VVV @VVV \\
S_{\la}(f) @> \subset > j > \Si_{\la} @> = >> \Si_{\la} @> g >> \R^n \x [0,1)
\end{CD}
\end{equation*}
where a lot of arrows are inclusions as it is denoted,
 $g$ is the immersion of $\Si_{\la}$ into $\R^n \x [0,1)$ and $\mathcal N (S_{\la}(f))$ is an  open tubular neighborhood of $S_{\la}(f)$.
The prefix ``$D$'' denotes the corresponding closed disk bundles.
The composition $\Phi_{\la} \circ i$ is the given $\va$-bundle of $f$.
The composition $g \circ j$ is the immersion $f|_{S_{\la}(f)}$.

Moreover the line bundle $\mathcal R_{\la}$ over $\Si_{\la}$, which is a trivial line bundle since 
we are working with framed fold maps,  is the normal bundle of the immersion $g \co \Si_{\la} \to \R^n \x [0,1)$.

So attaching the null-cobordisms to $f \x {\mathrm {id}}_{[0,3\ep)}$ we obtain
a framed fold map $$\tilde F \co V \to \R^n \x [0,1),$$ where $V = Q \x [0,3\ep) \cup_{1 \leq \la \leq (q+1)/2} \mathcal O_{\la}$ is a non-compact 
$(n+q+1)$-dimensional manifold with boundary $Q$. 

We also take the disk bundle $D\mathcal O_{\la}$ of $\mathcal O_{\la}$ and restrict the map $\Phi_{\la}$ to it.
In this way,
we obtain the restricted map $F \co W \to \R^n \x [0,1)$,
where 
\begin{enumerate}
\item
$W \subset V$ and
$W$ is a compact $(n+q+1)$-dimensional manifold 
obtained from attaching to $Q \x [0,2\ep]$ 
the spaces $$D\mathcal N\left(S_{\la}(f)\right) \x [2\ep, 3\ep]$$ and then
the (restricted) domains 
$D\mathcal O_{\la}$
of the null-cobordisms $\Phi_{\la}$ of the index $\geq 1$ fold singularity bundles of $f$,
\item
the boundary of $W$ is equal to 
$Q^{} \amalg P^{}$, where the
closed $(n+q)$-dimensional manifold $P^{}$ 
is diffeomorphic to the union of 
\begin{itemize}
\item
$Q - \bigcup_{1 \leq \la \leq (q+1)/2} D\mathcal N\left(S_{\la}(f)\right)$ and
\item
$\bigcup_{1 \leq \la \leq (q+1)/2} S\mathcal O_{\la}$,
where the prefix ``$S$'' denotes the corresponding sphere bundles,
\end{itemize}
\item
$F|_{Q^{}\x [0,\ep)} = f \x
{\mathrm {id}}_{[0,\ep)}$, 
where 
$Q^{} \x [0,\ep)$
 is a small collar neighborhood of $Q^{}$ ($\subset \del W^{}$) in $W^{}$ with the
identification $Q^{} = Q^{} \x \{0\}$ and
\item
$F$ is a restriction of $\tilde F$ and $\tilde F$ is a fold map with only definite fold singularities into $\R^n \x (0,1)$
near $P^{}$ and 
\item
the definite fold singularities of $F$ are of the form
$S_0(f) \x [0, 2\ep]$ in $Q \x [0,2\ep]$ and mapped by $$\left( f \x {\mathrm {id}}_{[0,2\ep]} \right) |_{S_0(f) \x [0, 2\ep]},$$
\end{enumerate}
see Figure~\ref{foldmap}.

\begin{figure}[ht]
\begin{center}
\epsfig{file=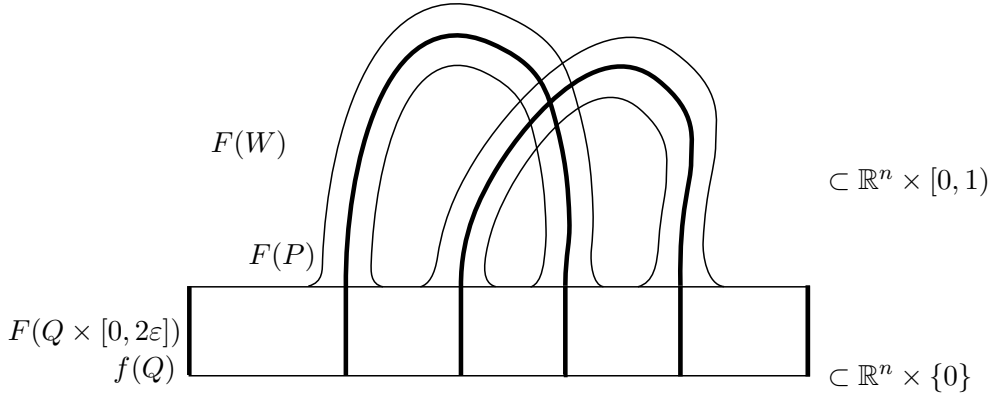, height=5cm}
\put(0.2, 2.5){$\subset \R^n \x [0,1)$}
\put(0.2, -0.1){$\subset \R^n \x \{ 0 \}$}
\put(-9.3, 0){$f(Q)$}
\put(-10.7, 0.5){$F(Q \x [0, 2\ep])$}
\put(-7.5, 1.5){$F(P)$}
\put(-8, 3){$F(W)$}
\end{center}   
\caption{The $F$-image of $W$ in $\R^n \x [0,1)$, where $F|_{Q \x \{ 0 \}} = f$ maps into $\R^n \x \{ 0 \}$. 
The thick segments and arcs represent 
the $F$-image of the fold singular set.}
\label{foldmap}
\end{figure}

Since $\tilde F|_{V}$ is a framed fold map, fixing a Riemannian metric $\varrho$ on $W$ we obtain a
$$\va(F, \varrho, r) \co TW \oplus \ep^1_{W} \to T\left( \R^n \x [0,1) \right)$$
fiberwise epimorphism for some small $r>0$ as in Section~\ref{existfrfold}.
This gives a stable $n$-framing on $W$.
We want this stably $n$-framed $W$ to be a cobordism between $Q$ and $P$ in the sense of Definition~\ref{stabframecob} but
while the framing of $W$ restricts to the framing of $Q$ as Definition~\ref{stabframecob} requires,
we do not get immediately a stable $(n-1)$-framing of $P$. This is because the framing of $W$ gives only 
$n+1$ linearly independent sections of
$$TP \oplus \ep^2_P$$ such that we do not know whether one of these sections is a normal section of $P$.
This problem occurs because $P$ is not mapped by $F$ into some $\R^n \x \{ t \}$, where $t \in \R$,  while $Q$ is.
But since we have $\pi_k\left(V_{n+1}\left(\R^{n+q+2}\right), V_n\left(\R^{n+q+1}\right)\right) = 0$
for the  relative homotopy groups of Stiefel varietes for $k \leq n+q$ (this follows from \cite[Chapter 8.11.]{Hu93}),
 there is a homotopy of these $n+1$ linearly independent sections of $TP \oplus \ep^2_P$
through linearly independent sections  such that at the end we obtain $n+1$ linearly independent sections of $TP \oplus \ep^2_P$
with the last section being parallel to the last $\ep^1_P$ summand. Deleting this we get
$n$ linearly independent sections of $TP \oplus \ep^1_P$
and $W$ becomes a cobordism between $Q$ and $P$ in the sense of Definition~\ref{stabframecob}.
Suppose this homotopy happens over $P \x [0,1]$. 
We identify our $P$ with $P \x \{ 0 \}$.
Then attach
this homotopy to $W$ along $P$. Hence $Q$ and $P \x \{1\}$ are stably $(n-1)$-framed cobordant by $W$ 
and the attached homotopy $P \x [0,1]$.  

It follows that 
since $\si_{n,q}([f])$ is zero, 
the stably $(n-1)$-framed manifold $P \x \{ 1 \}$ is also zero in the 
 cobordism group 
$\CC_{n+q}^{}(n)$. So by gluing a stably $(n-1)$-framed null-cobordism of
$P \x \{ 1 \}$ to $W \cup \left(P \x [0,1]\right)$ along $P \x \{ 1 \}$,
we obtain a
compact $(n+q+1)$-dimensional manifold $X$ with boundary
$Q$ such that the bundle $TX \oplus \ep^1_{X}$ 
has an $(n+1)$-framing which
coincides with the stable $n$-framing of $W$ over $W$. 

Since $\R^n \x [0,1]$ is contractible,
we can extend the map $F$ to
a continuous map $G \co X \to
\R^n \x [0,1]$.

Observe that if we introduce a Riemannian metric on $X$ extending the given Riemannian metric on $W$, then
the $(n+1)$-framing 
of the bundle $TX \oplus \ep^1_{X}$
gives a fiberwise epimorphism 
$$H \co TX \oplus \ep^1_{X} \to T(\R^n \x [0,1])$$ covering
the continuous map $G$
by mapping 
the $n+1$ frames to the standard bases of $T(\R^n \x [0,1])$
at any points of $X$. 
This $H$ coincides with $\va(F, \varrho, r)$ over $W$ because
we got $H$ by the framing determined by $\va(F, \varrho, r)$.

Hence
by Theorem~\ref{relexifold},
we see that 
there is a framed fold map $G' \co X \to
\R^n \x [0,1]$ 
which coincides with $F$ on 
a closed subset of
$W$ which contains some $Q \x \{t\}$ for some $t \in [0,2\ep]$. 
So the framed fold map $f$ is framed
null-cobordant. 
The oriented case is proved in a similar way.
\end{proof}

\subsection{Proof for  the cobordism group of oriented fold maps into the plane}\label{proof2}

In order to prove Theorem~\ref{explsikbathm} (and later Theorem~\ref{explsikbathmunori} in the unoriented case),
we need the following lemma. 

\begin{lem}\label{groupisomlemma}
Let $A, B, C$ be abelian groups and  
let $\iota_B \co B \to A \oplus B$ be the standard inclusion.
Let $\psi \co A \oplus B \to C$ be a surjective homomorphism and
let $\pi_B \co A \oplus B \to B$ be a homomorphism such that
$\pi_B \circ \iota_B$ is an isomorphism.
Suppose $( \psi , \pi_B) \co A \oplus B \to C \oplus B$ is injective.
%If $\psi \circ \iota$ is identically zero, then $( \psi , \pi)$ is surjective.
%Then
\begin{enumerate}[\rm (1)]
\item
If $\im (\psi \circ \iota_B) = 0$, then $( \psi , \pi_B)$ is also surjective so $A \oplus B$ is isomorphic to $C \oplus B$.
\item
If $\im (\psi \circ \iota_B) = \Z_2$ and $( \psi , \pi_B)$ is not surjective, then
$A \oplus B$ is isomorphic to $C / \im (\psi \circ \iota_B) \oplus B$.
\end{enumerate}
\end{lem}
\begin{proof}
Let $q \co C \to C / \im (\psi \circ \iota_B)$ denote the quotient map and let
$id_B \co B \to B$ denote the identity map of $B$.
We study the composition\footnote{If $G, G', H, H'$ are abelian groups and $\al \co G \to G'$ and 
$\be \co H \to H'$ are homomorphisms, then
let $\al \oplus \be \co G \oplus H \to G' \oplus H'$ denote the homomorphism which maps
$(g, h) \in G \oplus H$ to $(\al(g), \be(h))$.}
\begin{equation*}
\begin{CD}
A \oplus B @> ( \psi , \pi_B) >> C \oplus B @> q \oplus id_B >>   C / \im (\psi \circ \iota_B) \oplus B.   \\
\end{CD}
\end{equation*}
We show that it is surjective. 
Take an element $(x, y) \in C / \im (\psi \circ \iota_B) \oplus B$.
There is a $c \in C$ such that $q(c) = x$ and there is an $(a, b) \in A \oplus B$ such that $\psi (a, b ) = c$.
Then $\pi_B (a, b) =  b'$ for some $b' \in B$. 
Choose $b'' \in B$ such that $\pi_B \circ \iota_B (b'') = y - b'$.
This means that $\pi_B( 0, b'') = y - b'$.
Then $\psi (a, b + b'') = \psi(a, b) + \psi(0, b'')$ and $\pi_B (a, b + b'') = \pi_B (a, b) + \pi_B (0, b'')$,
so 
\begin{multline*}
q \oplus id_B  \circ ( \psi , \pi_B) (a, b + b'') = q \oplus id_B \left(   \psi(a, b) + \psi(0, b''), \pi_B (a, b) + \pi_B (0, b'') \right) = \\ 
q \oplus id_B \left(   \psi(a, b) + \psi(0, b''), y \right) =
\left(q\left(\psi(a, b)\right) + q\left(\psi(0, b'')\right), y\right) = (x, y)
\end{multline*}
because $q\left(\psi(0, b'')\right) = 0$. So $q \oplus id_B  \circ ( \psi , \pi_B)$ is surjective.
This immediately implies (1) since then $q \oplus id_B$ is the identity map.

To see that  (2) also holds note that if we have two groups $G$ and $H$, $H$ has a subgroup $\Z_2 \subset H$
and we have the composition
$$G \to H \to H / \Z_2,$$
where the first arrow, which we denote by $\al$, is injective but not surjective, the second arrow is the quotient map and the composition is surjective, then $G$ is isomorphic to $H / \Z_2$.
This is because $\al$ has to hit all the classes $h+\Z_2$  and if there is a $g \in G$ which goes to 
$1 \in \Z_2 \subset  H$, then for any $h \in H$ such that $h \in \im \al$ we have that  $h+1 \in \im \al$. So $\al$ would be surjective contradicting  our original assumption.
Hence $1 \in \Z_2$ is not in $\im \al$ so the composition is injective.
 \end{proof}

Now we prove Theorem~\ref{explsikbathm}.
\begin{proof}[Proof of Theorem~\ref{explsikbathm}]
Recall that by Theorem~\ref{invarithm} for $k \geq 1$
the homomorphism $\Im_{2,4k-2}^O$ is injective.
We will apply Lemma~\ref{groupisomlemma}.
By Proposition~\ref{directsumincobgroup} in 
Section~\ref{mobuimm} (see also \cite[Theorem~3.1 and Remark~3.3]{Kal7}) the group $\CC ob_{{}}^{O}(2, 4k-2)$
contains the group \[\bigoplus_{1 \leq j \leq 2k-1} \imm\left({\ep^1_{BS(O(j) \x O(4k-1-j))}},2\right)\]
as a direct summand, we denote this group by  $B$, so $\CC ob_{{}}^{O}(2, 4k-2) \cong A \oplus B$
for some group $A$. With these roles we will apply Lemma~\ref{groupisomlemma}.
The homomorphism $(\csi_{{1}}^O, \ldots, \csi_{{2k-1}}^O)$
will be $\pi_B$.

By Proposition~\ref{koschizom} the group ${\CC}_{4k}^O{(2)}$, which will be the group $C$, is isomorphic to the group
$\Omega_{4k}^{2|\chi}$ and the forgetful map 
${\CC}_{4k}^O{(2)} \to \Omega_{4k}^{2|\chi}$ which
forgets the framings gives an isomorphism.
Since by \cite{Lev1} every closed orientable manifold of dimension $> 2$ with even Euler characteristic has  a fold map into the plane,
the homomorphism $\si^O_{2,4k-2} \co \CC ob_{{}}^{O}(2, 4k-2)
 \to {\CC}_{4k}^O{(2)}$, which will be $\psi \co A \oplus B \to C$ when we apply Lemma~\ref{groupisomlemma}, is  surjective.
 
Now, let us apply  Lemma~\ref{groupisomlemma}. By case (1) we have, but also it is quite obvious, that  
to finish the proof of Theorem~\ref{explsikbathm}
 it is enough to show that 
% the contribution of the homomorphism
%$\bigoplus_{1 \leq j < (q+1)/2} \csi_{{j}}^O$ in the image ${\CC}_{2+q}^O{(2)} \cong \Omega_{2+q}^{2|\chi}$ of $\si^O_{2,q}$
%is zero.
%In other words, it is enough to show that
the source manifolds  of the fold maps
$\va_{j,4k-2, i_j} \co Q_{j,4k-2, i_j}^{4k} \to \R^2$ and
$\va_{j,4k-2, e_j} \co Q_{j,4k-2, e_j}^{4k} \to \R^2$, which are representatives of generators of the group $B$
(see Section~\ref{mobuimm}) represent zero in $\Omega_{4k}^{2|\chi}$. But the manifolds $Q_{j,4k-2,i_j}^{4k}$ and $Q_{j,4k-2,e_j}^{4k}$
are fibrations over the circle $S^1$ with fiber 
 the $(4k-1)$-dimensional sphere  with orientation preserving linear structure group, hence they are null-cobordant.
\end{proof}

\subsection{Proof of the Poincar\'e-Hopf type formula}\label{proof3}

Now, we prove Theorem~\ref{sign}.

\begin{proof}[Proof of Theorem~\ref{sign}]
%First, we prove the proposition for fold maps.
By Theorem~\ref{explsikbathm} the group $\CC ob_{{}}^{O}(2, 4k-2)$ is isomorphic to
$$\Omega_{4k}^{2|\chi} \oplus \Z_2^{4k-2}.$$
We check the values 
$$[f_{\ga}|_{S_f}] + \sum_{j=1}^{2k-1} \tau_j(f_{\ga})$$ 
and $$\frac{\si( M_{\ga}^{4k} )}{2} \mod{2}$$
for a system of generators $\{ [f_{\ga} \co M_{\ga}^{4k} \to \R^2] \}_{\ga \in \Ga}$ 
of the cobordism group $\CC ob_{{}}^{O}(2,4k-2)$. 
%By Proposition~\ref{explsikbathm} the group $\CC ob_{{}}^{O}(2, 4k-2)$ is isomorphic to
%$\Omega_{4k}^{2|\chi} \oplus \Z_2^{4k-2}$. 
A system of representatives of the generators of the 
$\Z_2^{4k-2}$
summand 
of $\CC ob_{{}}^{O}(2,4k-2)$
is given by the fold maps $\va_{j,4k-2, i_j}$ and
$\va_{j,4k-2, e_j}$, where $[i_j]$ and $[e_j]$ are the generators
of the group $$\imm\left({\ep^1_{BS(O(j) \x O(4k-1-j))}},2\right) \cong \Z_2^2,$$ 
$1\leq j \leq 2k-1$,
see Section~\ref{mobuimm}. Note that the source manifolds of 
$\va_{j,4k-2, i_j}$ and
$\va_{j,4k-2, e_j}$ are oriented null-cobordant  for $1\leq j \leq 2k-1$.

To generate the direct summand  $\Omega_{4k}^{2|\chi}$ of $\CC ob_{{}}^{O}(2,4k-2)$ as well, we 
construct a fold map into the plane 
of each element of a system of representatives
of generators
 of the group $\Omega_{4k}^{2|\chi}$ as follows.
Let us take a class $\omega \in \Omega_{4k}^{2|\chi}$. It can be written in the form 
$$\omega = r[\CP^{2k} \# \CP^{2k}] + [Z^{4k}],$$ where $r \geq 0$, the manifold $\CP^{2k} \# \CP^{2k}$ is oriented in {\it some} way and the signature $\si( Z^{4k} )$ is equal to zero.
By \cite[Theorem~3]{AK}, we can suppose that $Z^{4k}$ is a fiber bundle over $S^2$ with a closed orientable 
$(4k-2)$-dimensional manifold $F^{4k-2}$ as fiber.

\begin{prop}\label{fiberedmaprelOK}
There is an oriented fold map $f_z \co Z \to \R^2$ such that
$$t(f_z)  \equiv 0 \mod{2}$$ and
$$\tau(f_z) \equiv 0 \mod{2}.$$
\end{prop}
\begin{proof}
Clearly there exists a Morse function $\mu \co F^{4k-2} \x [0,1] \to [0,1]$ such that 
$\mu$ has no singularities near $F^{4k-2} \x \{0\} \amalg F^{4k-2} \x \{1\}$,
$\mu$ takes its maximum $1$ on $F^{4k-2} \x \{0\} \amalg F^{4k-2} \x \{1\}$,
 and
$\mu^{-1}(1-t) = F^{4k-2} \x \{t\} \amalg F^{4k-2} \x \{1-t\}$ for $0 \leq t \leq 1/4$.

Let $\tilde f_z$ be the fold map of $F^{4k-2} \x [0,1] \x S^1$ into $\R^2$ defined by
$$\tilde f_z(u, s, e^{i\theta}) = (2-\mu(u,s))e^{i\theta}$$  for $u \in F^{4k-2}, 
s \in [0,1], \theta \in [0,2\pi]$, where we identify $S^1$ with the unit circle $\{ e^{i\theta} :  \theta \in [0,2\pi] \}$.
So we have the fold map
$$\tilde f_z \co F^{4k-2} \x [0,1] \x S^1 \to \{ 1 \leq |x| \leq 2 \} \subset \R^2,$$
which maps the 
boundary $\left( F^{4k-2} \x \{0\} \amalg F^{4k-2} \x \{1\} \right) \x S^1$ into the unit circle $\{|x|=1\}$.

Let $D^2_+$ and $D^2_-$ denote the northern and southern hemispheres of $S^2$, respectively.
According to the decomposition $$D^2_+ \bigcup_{} D^2_- = S^2,$$ %where $\al \co \del D^2_+  \to \del D^2_-$ is the standard gluing diffeomorphism,
the bundle $F^{4k-2} \hookrightarrow Z^{4k} \to S^2$ falls into two pieces, i.e.\ into the trivial $F^{4k-2}$ bundles
$p_+ \co Z^{4k}_+ \to D^2_+$ and $p_- \co Z^{4k}_- \to D^2_-$.
Let us identify $D^2_{+}$ with the standard unit disk $\{|x| \leq 1\}$
and $D^2_-$ with the disk   $\{|x|\leq 2\}$ of radius $2$ in $\R^2$.

We define the fold map $$f_z \co Z^{4k} \to \R^2$$ by 
\begin{itemize}
\item
$f_z|_{Z^{4k}_+} = p_+$, 
\item
$f_z|_{p_-^{-1}(\{ |x| \leq 1\})} = p_-$ and
\item
$f_z|_{p_-^{-1}( \{ 1 \leq |x| \leq 2 \})} = \tilde f_z$, where
$p_-^{-1}( \{ 1 \leq |x| \leq 2 \}) = F^{4k-2} \x \{ 1 \leq |x| \leq 2 \}$ and
the annulus $\{ 1 \leq |x| \leq 2 \}$ is identified with $[0,1] \x S^1$ 
\end{itemize}
so that
the resulting map $f_z$ is a fold map.

\begin{lem}
We have
$[f_{z}|_{S_{f_z}}] \equiv 0 \mod{2}$ and
 $\sum_{j=1}^{2k-1} \tau_j (f_z)  \equiv 0 \mod{2}$.
\end{lem}
\begin{proof}
It is easy to see that the $f_z$-image of the fold singular set of $f_z$ consists of
concentric circles, moreover each determinant bundle $\delta_j(f_z)$, see Remark~\ref{detnyaltwist}, is  trivial
if the set $S_j(f_z)$ is non-empty, for $1\leq j \leq 2k-1$.
\end{proof}

This completes the proof of Proposition~\ref{fiberedmaprelOK}.
\end{proof}

\begin{prop}\label{prop2CP^2-nfoldmapproof-al}
Let $k \geq 1$.
There is an oriented fold map $f \co \CP^{2k} \# \CP^{2k} \to \R^2$
such that $$t(f)  \equiv 0 \mod{2}$$ and
$$\tau(f) \equiv 1 \mod{2}.$$
\end{prop}
\begin{proof}
Let us orient $\CP^{2k}$ in the standard way.
We define the fold map $f_{C} \co \CP^{2k} \# \CP^{2k} \to \R^2$ as follows.
By \cite{Lev1} there is a stable map $$g \co \CP^{2k} \to \R^2$$ with only one cusp point 
$p \in \CP^{2k}$ and also we know that the index of the fold singularities around this cusp is $2k-1$.
Our plan is to ``eliminate'' these two cusps in the ``two copies'' of $\CP^{2k}$ in the connected sum 
$\CP^{2k} \# \CP^{2k}$.

Since the singular set of $g$ is connected, we can suppose that there is an embedded arc $a \co [0,1] \to \R^2$
% transversal on $(0,1)$ to the image $g(S_g)$ of the singular set $S_g$ 
such that $a(0)= g(p)$, $a(1) \in \R^2 - g(\CP^{2k})$ and
$a\left( (0,1) \right)$ intersects the image of the singular set of $g$ transversally and exactly at one definite fold value.
%Let $d \in (0,1)$ be that point for which $a(d)$ is equal to that definite fold value.

Take a small tubular neighborhood of $a([0,1])$ in $\R^2$. Then the boundary $C$ of this neighborhood 
is a circle embedded into $\R^2$ which
divides $g(\CP^{2k})$
into two regions. (Of course $C \cap g(\CP^{2k})$ is an embedded interval in $\R^2$.)
One of them contains $g(p)$ and the other does not. Denote the region containing $g(p)$ by $R$, see Figure~\ref{complexproj}.

\begin{figure}[ht]
\begin{center}
\epsfig{file=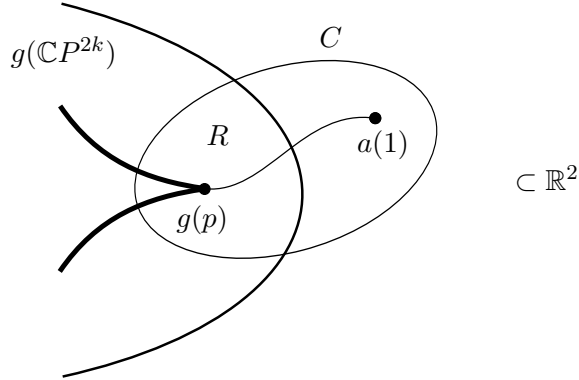, height=5cm}
\put(1, 2.5){$\subset \R^2$}
\put(-5.7, 4.2){$g(\CP^{2k})$}
\put(-3.5, 2){$g(p)$}
\put(-1.1, 3){$a(1)$}
\put(-1.6, 4.4){$C$}
\put(-3.1, 3.1){$R$}
\end{center}   
\caption{The $g$-image of $\CP^{2k}$ in $\R^2$. The thick arcs represent the $g$-image of the fold singular set going into 
the cusp value $g(p)$. The arc connecting $g(p) = a(0)$ and $a(1)$ intersects transversally the $g$-image of the definite fold singular set.}
\label{complexproj}
\end{figure}

Then the preimage $g^{-1}(R)$ is an embedded $4k$-dimensional ball in $\CP^{2k}$.
Moreover 
the $g$-preimage of $C \cap g(\CP^{2k})$ is the boundary of this $4k$-dimensional ball, which is an $S^{4k-1}$.
And if we identify $C \cap g(\CP^{2k})$ with $\R$, then the map
$g$ restricted to this $S^{4k-1}$ is a Morse function with four critical points: two definite critical points,
one critical point of index $2k-1$ and one critical point of index $2k$.
This Morse function gives a handle decomposition of $S^{4k-1}$. 
Since the two middle critical points form a cancelling pair (because of the cusp point), the $(2k-1)$-handle 
is attached to the $0$-handle such that the attaching sphere is the standard $(2k-2)$-dimensional 
sphere with trivial framing. 
Watching the Morse function ``upside down'' we have the same thing about the other handles.

So after identifying $g^{-1}(C \cap g(\CP^{2k}))$ with $S^{4k-1}$ and 
$C \cap g(\CP^{2k})$ with $\R$, we get the Morse function $g|_{g^{-1}(C \cap g(\CP^{2k}))}$, which we denote by $h$.
By the previous argument, 
we can suppose that $$S^{4k-1} = \{ |(x_1,\ldots x_{4k})|=1\}$$ 
is given
in the form $$S^{2k-1} \x D^{2k} \bigcup_{}  D^{2k} \x S^{2k-1},$$ where $S^{2k-1} \x D^{2k}$ and $D^{2k} \x S^{2k-1}$ are identified with the subsets
$\{x_{2k+1}^2 + \cdots + x_{4k}^2 \leq 1/2\}$ and $\{ x_{1}^2 + \cdots + x_{2k}^2 \leq 1/2\}$ of $\{ |(x_1,\ldots x_{4k})|=1\}$, respectively.
We can suppose, that this decomposition coincides with the handle decomposition of $S^{4k-1}$ given by the Morse function $h$, $h^{-1}(0)=\{ x_{1}^2 + \cdots + x_{2k}^2 = 
 x_{2k+1}^2 + \cdots + x_{4k}^2 = 1/2\}$, 
 the critical point of index $2k-1$ is
 $(1,0,\ldots,0) \in S^{2k-1} \x \{0\}$,
 and  the critical point of index $2k$ is in $\{0\} \x S^{2k-1}$ with $(2k+1)$-st coordinate equal to $1$ and other coordinates equal to $0$.

Now let us take the map $g \amalg (T \circ g) \co  \CP^{2k} \amalg \CP^{2k} \to \R^2$, where $T \co \R^2 \to \R^2$ 
is an affine translation such that the image of $T \circ g$ is disjoint from the image of $g$.
Since $T \circ g \co \CP^{2k} \to \R^2$
is just a copy of the map $g$, we also get a copy of the Morse function $h$ applying all the previous constructions 
to $T \circ g$ instead of $g$.
Roughly speaking, the Morse function $h$ is 
$$g|_{g^{-1}\left(C \cap g\left(\CP^{2k}\right)\right)}$$ and this other Morse function
is 
$$T \circ g|_{\left(T \circ g\right)^{-1}\left(T(C) \cap T \circ g\left(\CP^{2k}\right)\right)},$$
which can be naturally identified with $h$.

Now, we want to form the connected sum $\CP^{2k} \# \CP^{2k}$ and obtain a map
$$f_C \co \CP^{2k} \# \CP^{2k} \to \R^2,$$
which ``coincides'' with $g|_{g^{-1}(\R^2 - R)}$ on the first $\CP^{2k}$ summand and
with $T \circ g|_{(T \circ g)^{-1}(\R^2 - T(R))}$ on the second $\CP^{2k}$ summand.
All we need is an automorphism 
$$\left(\va \co S^{4k-1} \to S^{4k-1} , \psi \co \R \to \R \right)$$ of the Morse function 
$h \co S^{4k-1} \to \R$, where
both of $\va$ and $\psi$ reverse the orientation. If we have this automorphism, we
can take $g|_{g^{-1}(\R^2 - R)}$ and 
$T \circ g|_{(T \circ g)^{-1}(\R^2 - T(R))}$
and glue them together along the two Morse functions by this automorphism. 
Then we get $f_C \co \CP^{2k} \# \CP^{2k} \to \R^2$, see Figure~\ref{complexproj2}.

\begin{figure}[ht]
\begin{center}
\epsfig{file=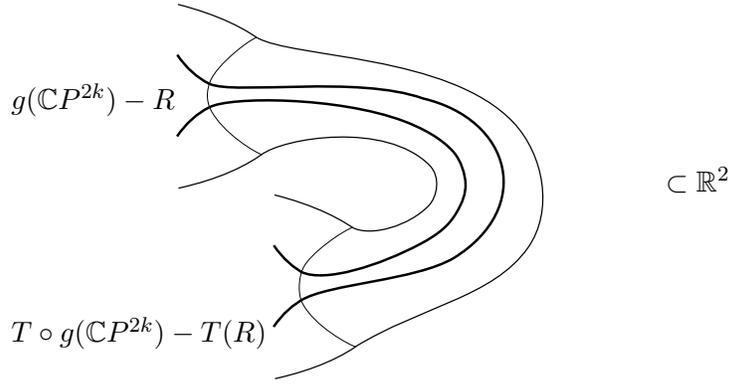, height=5cm}
\put(1.6, 2.5){$\subset \R^2$}
\put(-7.1, 3.6){$g(\CP^{2k}) - R$}
\put(-7.1, 0.5){$T \circ g(\CP^{2k}) - T(R)$}
\end{center} 
\caption{The $f_C$-image of $\CP^{2k} \# \CP^{2k}$ in $\R^2$ around the $f_C$-image of the  ``connecting tube
between the two $\CP^{2k}$ summands''. The thick arcs represent the $f_C$-image of the fold singular set.
The gluing of the ``connecting tube'' realizes the automorphism $(\va, \psi)$.}
\label{complexproj2}  
\end{figure}

We define the diffeomorphism $\va \co S^{4k-1} \to S^{4k-1}$ to be induced by the linear transformation
$$(x_1,\ldots, x_{4k}) \mapsto (x_{2k+1}, \ldots, x_{4k}, x_1, \ldots, x_{2k-1}, -x_{2k}).$$
Clearly $\va$ interchanges the critical points of indices $2k-1$ and $2k$,
maps the unstable and stable manifolds into the stable and unstable manifolds respectively,
while reversing the orientations of the unstable manifolds. 

Hence we get $f_C \co \CP^{2k} \# \CP^{2k} \to \R^2$, which is a stable fold map.

\begin{lem}\label{2CP^2-nfoldmap}
We have 
$[f_{C}|_{S_{f_C}}] \equiv 0 \mod{2}$ and
$\sum_{j=1}^{2k-1}  \tau_j (f_C)  \equiv 1  \mod{2}$.
\end{lem}
\begin{proof}
Since we obtained the fold map $f_C$ from two copies of the map $g$, 
and the performed operations did not change the number of double points 
of the image of the singular sets,
 it is clear that
$[f_{C}|_{S_{f_C}}] \equiv 0 \mod{2}$. 

We also have that $\va$ interchanges the critical points of indices $2k-1$ and $2k$,
maps the unstable and stable manifolds into the stable and unstable manifolds respectively,
while reversing the orientations of the unstable manifolds. Hence
the twisting modulo $2$ of the indefinite fold germ bundle is $1$ over the component of $S_{f_C}$
which we obtain after constructing $f_C$ and which goes through the ``connecting tube'' of $\CP^{2k} \# \CP^{2k}$. 
The twisting modulo $2$ equal to $0$ over the union of the other components, because
each of the other components of $S_{f_C}$ in the first 
$\CP^{2k}$ summand 
has an identical pair in the second  $\CP^{2k}$ summand. 
\end{proof}
This lemma completes the proof of Proposition~\ref{prop2CP^2-nfoldmapproof-al}.
\end{proof}

Recall the definition of the fold maps $\va_{j,4k-2, i_j}$ and $\va_{j,4k-2, e_j}$ in Section~\ref{mobuimm}.
\begin{lem}
The fold maps 
$\va_{j,4k-2, i_j}$ and
$\va_{j,4k-2, e_j}$, where $[i_j]$ and $[e_j]$ are the generators
of the group $\imm\left({\ep^1_{BS(O(j) \x O(4k-1-j))}},2\right) \cong \Z_2^2$, 
satisfy for $1\leq j \leq 2k-1$ the congruence $\frac{\si}{2} \equiv t + \tau \mod{2}$
of Theorem~\ref{sign}.
\end{lem}
\begin{proof}
Concerning the signature:
The source manifolds of $\va_{j,4k-2, i_j}$ and
$\va_{j,4k-2, e_j}$ are null-cobordant, since they are sphere bundles with linear structure groups.
Hence $${\si( Q_{j,4k-2, i_j}^{4k} )}= {\si( Q_{j,4k-2, e_j}^{4k} )} = 0.$$

Concerning the double points:
By construction, $\va_{j,4k-2, e_j}$ restricted to its fold singular set 
has $1$ double point if $j=1$ (the double point of the definite fold singular set if we perturb a little),
and has no double point if $j \geq 2$.
 Clearly $\va_{j,4k-2, i_j}$ restricted to its singular set has even number of double points for each 
 $1\leq j \leq 2k-1$.
 
Concerning the twisting of the normal bundle of the singular set:
The  determinant bundle $\delta_l(\va_{j,4k-2, i_j})$ of the $O(l)$
bundle obtained by the projection $$S\left(O(l) \x O(4k-1-l)\right) \to O(l)$$ over $S_l(\va_{j,4k-2, i_j})$
is trivial ($l=j-1, j$ and $l \geq 1$), the
 determinant bundle $\delta_1(\va_{1,4k-2, e_1})$ is non-trivial, and for $j \geq 2$
 the
 determinant bundle $\delta_l(\va_{j,4k-2, e_j})$ is non-trivial for $l=j$ and $l=j-1$, see Remark~\ref{symm} in Section~\ref{mobuimm}.
 
 So if we sum up the double points for the index $0, \ldots, 2k-1$ fold singularities and 
the twistings of index  $1, \ldots, 2k-1$ fold singularities, we get 
\begin{itemize}
\item
zero for $\va_{j,4k-2, i_j}$, $1\leq j \leq 2k-1$, because we have 
even number of double points and trivial twisting,
\item
zero for $\va_{1,4k-2, e_1}$, because we have one double point (from one definite fold crossing after perturbation) and non-trivial twisting (the twisting of the index $1$ fold singular set),
\item
zero for $\va_{j,4k-2, e_j}$,  $2 \leq j \leq 2k-1$, because we have no double points and we have non-trivial twistings for both of the 
index $j-1$ and $j$ fold singularities (and two non-trivial values are zero together).
\end{itemize}
And also all the source manifolds have zero signature. This completes the proof.
 \end{proof}

Now, for an arbitrary oriented  fold map $f \co M^{4k} \to \R^2$, let us write the cobordism class $[M^{4k}]$ in the form $r[\CP^{2k} \# \CP^{2k}] + [Z^{4k}]$ as we explained before, and the cobordism class
$[f]$ in the form $r[f_C] + [f_z] + 
\sum_{1\leq j \leq 2k-1} a_j [\va_{j,4k-2, i_j}] + b_j [\va_{j,4k-2, e_j}]$, where $a_j, b_j \in \{0,1\}$.
By the above, each summand satisfies the congruence $\frac{\si}{2} \equiv t + \tau \mod{2}$ in the statement of Theorem~\ref{sign}, which
completes the proof of Theorem~\ref{sign}.
\end{proof}

\subsection{Proof for the unoriented case}\label{proof4}

\begin{lem}\label{fiberedmaprelOKunori}
Let $Z$ be a closed (possibly non-orientable) %$m$-dimensional 
manifold which fibers over $S^2$.
Then there exists a fold map $f_z \co Z \to \R^2$ such that
$$t(f_z)  \equiv 0 \mod{2}$$ and
$$\tau^1(f_z) \equiv \tau^2(f_z) \equiv 0 \mod{2}.$$ 
%and
%$$\tilde \tau(f_z) \equiv 0 \mod{2}.$$
\end{lem}
\begin{proof}
The proof is completely analogous to the proof of Proposition~\ref{fiberedmaprelOK}. %but instead of \cite{AK} we use \cite{Br69}. 
Details are left to the reader.
\end{proof}

Recall the definition of the fold maps $\tilde \va_{j,2k-2, i_j}$, $\tilde \va_{j,2k-2, e_j^1}$ and  $\tilde \va_{j,2k-2, e_j^2}$ in Section~\ref{mobuimm}.

\begin{lem}\label{immgencong}
For $1\leq j \leq k-1$
the fold maps 
$\tilde \va_{j,2k-2, i_j}$,
$\tilde \va_{j,2k-2, e_j^1}$ and  $\tilde \va_{j,2k-2, e_j^2}$, where $[i_j]$, $[e_j^1]$ and $[e_j^2]$ are the generators
of the group $\imm\left({\ep^1_{B(O(j) \x O(2k-1-j))}},2\right) \cong \Z_2^3$, satisfy 
$$\tau^1 + \tau^2 \equiv 0 \mod{2}.$$
\end{lem}
\begin{proof}
The twistings of the maps $\tilde \va_{j,2k-2, i_j}$ are trivial. %and the number of double points of the immersion of their singular set is even.

For the map $\tilde \va_{1,2k-2, e_1^1}$ we have 
non-trivial $\tau^1_1$-twisting (for the index one singular set), trivial $\tau^2_1$-twisting (again for the index one singular set) and
odd number of twisting for the index zero singular set. 
%and one double point for the immersion of the index zero singular set (and no double points for the immersions of 
%the singular sets of other index). 
So this map also satisfies $\tau^1 + \tau^2 \equiv 0 \mod{2}$.

For the map $\tilde \va_{1,2k-2, e_1^2}$ we have non-trivial $\tau^2_1$-twisting, trivial $\tau^1_1$-twisting and non-trivial twisting for the index zero singular set (three times non-trivial). %and no double points at all. 
So this map also satisfies $\tau^1 + \tau^2 \equiv 0 \mod{2}$.

The other maps $\tilde \va_{j,2k-2, e_j^1}$ (resp.\ $\tilde \va_{j,2k-2, e_j^2}$), where $j \geq 2$,
have non-trivial $\tau^1_j$-twisting and non-trivial $\tau^1_{j-1}$-twisting (resp.\ non-trivial $\tau^2_j$-twisting and non-trivial $\tau^2_{j-1}$-twisting) and no other twistings.
So they also satisfy $\tau^1 + \tau^2 \equiv 0 \mod{2}$.
\end{proof}

Now we prove Theorem~\ref{explsikbathmunori}.

\begin{proof}[Proof of Theorem~\ref{explsikbathmunori}]
We use Lemma~\ref{groupisomlemma} again but in a more sophisticated way.
By Proposition~\ref{directsumincobgroupunori} for $k \geq 1$, $q=4k-2$ the group
$\CC ob_{{}}(2, q)$ contains the group 
$$\bigoplus_{1 \leq j \leq q/2}  \imm\left({\ep^1_{B(O(j) \x O(q+1-j))}},2\right) \cong \Z_2^{6k-3}$$
as a direct summand. This direct summand will be denoted by $B$, so $\CC ob_{{}}(2, q) \cong A \oplus B$
for some group $A$.

For $C$ we take the group ${\CC}_{2+q}{(2)}$ and for $\psi$ we take the homomorphism
  $$\si_{2,q} \co \CC ob_{{}}(2, q)  \to {\CC}_{2+q}{(2)}.$$
  
Then  $\psi$ is surjective because by Remark~\ref{framinghomot}
 on any stably $1$-framed manifold $M$ there is a fold map $g$ into $\R^2$ such that 
the decomposition $TM \oplus \ep^1_M = \eta_M \oplus \ep^2_M$ induced by $g$
is cobordant to the given stable $1$-framing of $M$. 

The homomorphism 
$(\csi_{1}, \ldots, \csi_{q/2})$
 will play the role of $\pi_B$.

We want to show that $( \psi , \pi_B)$ is not only injective as Theorem~\ref{invarithm} says but also surjective.

Now, suppose that $( \psi , \pi_B)$ is not surjective. If we show that $\im (\psi \circ \iota_B) = \Z_2$ and this $\Z_2$
is the last $\Z_2$ summand in ${\CC}_{2+q}{(2)} \cong {\mathfrak N}_{4k}^{2|\chi} \oplus \Z_2$, then 
by applying case (2) of Lemma~\ref{groupisomlemma} we get that 
for $k \geq 1$, the group  $\CC ob_{{}}(2, 4k-2)$
is isomorphic to $${\mathfrak N}_{4k}^{2|\chi}    \oplus \Z_2^{6k-3}$$ 
since then the $\Z_2$ summand of ${\CC}_{2+q}{(2)} \cong {\mathfrak N}_{4k}^{2|\chi} \oplus \Z_2$ is factored out
by the quotient map $q$ (we keep the notations of Lemma~\ref{groupisomlemma}).
%Then  
%an isomorphism is given by the map
%$$[f \co M^{2k} \to \R^2] \mapsto \left( [M^{2k}] ,   
%t_1(f),  \tau_{1}^1(f), \tau_{1}^2(f), \ldots, t_{k-1}(f),  \tau_{k-1}^1(f), \tau_{k-1}^2(f) \right).$$
This will lead to a contradiction as we will see later, so $( \psi , \pi_B)$ is  surjective and since it was injective, we will get the proof of the statement.

So we show that $\im (\psi \circ \iota_B) = \Z_2$.
By Proposition~\ref{directsumincobgroupunori}, similarly to the oriented case, 
the  source manifolds $\tilde Q_{j,q, i_j}^{2+q}$, $\tilde Q_{j,q, e_j^1}^{2+q}$ and $\tilde Q_{j,q, e_j^2}^{2+q}$ of  the 
fold maps $\tilde \va_{j,q, i_j}$,
$\tilde \va_{j,q, e_j^1}$ and $\tilde \va_{j,q, e_j^2}$, respectively,
(see Section~\ref{mobuimm}) represent zero in the cobordism group ${\mathfrak N}_{2+q}$.
But the singular set of $\tilde \va_{1,q, e_1^1} \co \tilde Q_{1,q, e_1^1}^{2+q} \to \R^2$
is immersed with exactly one double point  into $\R^2$.
By Proposition~\ref{koschizomexplained}
this
 gives the non-trivial element in the direct summand $\Z_2$ of
${\CC}_{2+q}{(2)} \cong  {\mathfrak {N}}_{2+q}^{2|\chi} \oplus \Z_2$. Hence $\im (\psi \circ \iota_B) = \Z_2$.

The remaining fact to show is that $$\CC ob_{{}}(2, 4k-2) \cong {\mathfrak N}_{4k}^{2|\chi}    \oplus \Z_2^{6k-3}$$ cannot hold.
%If $\CC ob_{{}}(2, 2k-2) \cong {\mathfrak N}_{2k}^{2|\chi}    \oplus \Z_2^{3k-3}$ was true, then
%by Proposition~\ref{prop2CP^2-nfoldmap} and by the commutative diagram 
Consider the commutative diagram
\begin{center}
\begin{graph}(10,4)
\graphlinecolour{1}\grapharrowtype{2}
\textnode{D}(0.5, 3){$\Omega_{4k}^{2|\chi} \oplus \Z_2^{4k-2}$}
\textnode{E}(5.5, 3){$\left( {\mathfrak N}_{4k}^{2|\chi} \oplus \Z_2 \right)  \oplus \Z_2^{6k-3}$}
\textnode{F}(10.5, 3){${\mathfrak N}_{4k}^{2|\chi}  \oplus \Z_2^{6k-3}$}
\textnode {A}(0.5,1.5){$\CC ob_{{}}^O(2,4k-2)$}
\textnode {B}(5.5, 1.5){$\CC ob_{{}}(2,4k-2)$}
\textnode {C}(3, 0){$\Z_2$}
\diredge {A}{B}[\graphlinecolour{0}]
\diredge {B}{C}[\graphlinecolour{0}]
\diredge {A}{C}[\graphlinecolour{0}]
\diredge {A}{D}[\graphlinecolour{0}]
\diredge {B}{E}[\graphlinecolour{0}]
\diredge {D}{E}[\graphlinecolour{0}]
\diredge {E}{F}[\graphlinecolour{0}]
\diredge {B}{F}[\graphlinecolour{0}]
\freetext (1.3, 2.2){$\Im_{2,4k-2}^O$}
\freetext (4.7, 2.2){$\Im_{2,4k-2}$}
\freetext (8.2, 3.3){$q \oplus {\mathrm {id}}$}
\freetext (3,1.8){$\iota$}
\freetext (2.7,3.3){$\kappa$}
\freetext (1.1, 0.6){$\tau$}
\freetext (4.9, 0.6){$\tilde \tau$}
\freetext (8.2, 2.0){$p$}
\end{graph}
\end{center}
where $\kappa$ is the natural map corresponding to the natural map 
\begin{multline*}
{\CC}_{2+q}^O{(2)} \oplus  \bigoplus_{1 \leq j \leq q/2}  \imm\left({\ep^1_{BS(O(j) \x O(q+1-j))}},2\right) 
\lra \\  {\CC}_{2+q}{(2)} \oplus \bigoplus_{1 \leq j \leq q/2}  \imm\left({\ep^1_{B(O(j) \x O(q+1-j))}},2\right)
\end{multline*}
and $p$ is the composition $\left( q \oplus {\mathrm {id}} \right) \circ \Im_{2,4k-2}$.
Note that the arrows $\Im_{2,4k-2}^O$ and $\Im_{2,4k-2}$ are injective (and we know already that $\Im_{2,4k-2}^O$ is an isomorphism).

Take 
the cobordism class of the map $$f_C \co \CP^{2k} \# \CP^{2k} \to \R^2$$ of Proposition~\ref{prop2CP^2-nfoldmapproof-al}
in $\CC ob^O(2, 4k-2)$. 
Let us check the coordinates of $\Im_{2,4k-2}( \iota([f_C])  )$. Since $t(\iota([f_C])) = 0$ and $\CP^{2k} \# \CP^{2k}$ is null-cobordant 
in ${\mathfrak N}_{4k}^{2|\chi}$, we have that $\Im_{2,4k-2}( \iota([f_C])  )$ can have non-zero coordinates only in the direct summand $\Z_2^{6k-3}$.
Also $p( \iota([f_C])  )$ can have non-zero coordinates only in the direct summand $\Z_2^{6k-3}$.
So if $V$ denotes $$\Im_{2,4k-2} \circ \iota \left( \CC ob^O(2, 4k-2) \right) \cap \Z_2^{6k-3}$$ and
$W$ denotes
$$p \circ \iota \left( \CC ob^O(2, 4k-2) \right) \cap \Z_2^{6k-3},$$ then
$$\Im_{2,4k-2}( \iota([f_C])  ) \in V$$ and
$$p( \iota([f_C])  ) \in W.$$

Since $\tau([f_C]) = 1$, we have $\tilde \tau (\iota([f_C])) = 1$.
Now we want to compute $\tilde \tau (\iota([f_C]))$ in another way.
We know that the map $\tilde \tau$ is a homomorphism if we restrict it to the image of $\iota$.
So if we find elements $a_1, \ldots, a_l$ in $\CC ob^O(2, 4k-2)$ whose $\Im_{2,4k-2} \circ \iota$-image generates exactly 
the subspace $V$% = \Im_{2,4k-2} \circ \iota \left( \CC ob^O(2, 4k-2) \right) \cap \Z_2^{6k-3}$$ 
of 
the
direct summand $\Z_2^{6k-3}$, then
writing $\iota([f_C])$ as an appropriate linear combination of $\iota (a_1), \ldots, \iota(a_l)$ we could compute 
$\tilde \tau (\iota([f_C]))$ just by taking that linear combination of the values $\tilde \tau ( \iota ( a_1)), \ldots, \tilde \tau (\iota(a_l))$.
The classes of fold maps $\va_{j,4k-2, i_j}$, $\va_{j,4k-2, e_j}$, which generate 
 the direct summand $\Z_2^{4k-2}$ of
$\Omega_{4k}^{2|\chi} \oplus \Z_2^{4k-2}$ would be a natural choice for such elements $a_1, \ldots, a_l$ but the problem is that
$\Im_{2,4k-2} \circ \iota ([\va_{1,4k-2, e_1}])$ has a non-zero coordinate in the $\Z_2$ summand of ${\mathfrak N}_{4k}^{2|\chi} \oplus \Z_2$ so 
the $\Im_{2,4k-2} \circ \iota$-images of all the $[\va_{j,4k-2, i_j}]$ and $[\va_{j,4k-2, e_j}]$ do not generate exactly the 
subspace $V$ of the 
direct summand $\Z_2^{6k-3}$ in
$\left( {\mathfrak N}_{4k}^{2|\chi} \oplus \Z_2 \right)  \oplus \Z_2^{6k-3}$.

But we suppose now (in order to find a contradiction) that the arrow $p$ is an isomorphism.
And since the $p \circ \iota$-images of all the $[\va_{j,4k-2, i_j}]$ and $[\va_{j,4k-2, e_j}]$ generate exactly the 
subspace $W$ of the 
direct summand $\Z_2^{6k-3}$, a linear combination of the 
$\iota([\va_{j,4k-2, i_j}])$, $\iota([\va_{j,4k-2, e_j}])$s has to give $\iota ([f_C])$ and 
the same linear combination of the 

$$\tilde \tau ( \iota ( [\va_{j,4k-2, i_j}] )), \tilde \tau (\iota( [\va_{j,4k-2, e_j}]))\mathrm s$$ has to give $1$ (because $\tilde \tau (\iota([f_C])) = 1$).
In this linear combination $\tilde \tau (\iota( [\va_{1,4k-2, e_1}]))$ has to participate with non-zero coefficient because $\tau([\va_{1,4k-2, e_1}]) = 1$ 
while the other $\tau ([\va_{j,4k-2, e_j}])$ and $\tau([\va_{j,4k-2, i_j}])$ values are $0$.
But this leads to a contradiction because then $t([f_C]) = 1$ would hold since $t([\va_{1,4k-2, e_1}]) = 1$ while the other 
$t([\va_{j,4k-2, e_j}])$ and $t([\va_{j,4k-2, i_j}])$ values are $0$.
Since we know that $t([f_C]) = 0$, we have the contradiction, so the proof is finished.
\end{proof}

\begin{lem}\label{Z2partgen}
The classes
$$\Im_{2,4k-2}( [e_j^1]),  \Im_{2,4k-2}( [e_j^2]),  \Im_{2,4k-2}( [i_j])$$
for $1 \leq j \leq 2k-1$
and $\Im_{2,4k-2}( [f_C]  )$
form a basis of the  $\Z_2   \oplus \Z_2^{6k-3}$ direct summand  of $\CC ob_{{}}(2, 4k-2) \cong  {\mathfrak N}_{4k}^{2|\chi} \oplus \Z_2   \oplus \Z_2^{6k-3}$.
\end{lem}
\begin{proof}
This follows from the previous proof and the constructions in Section~\ref{mobuimm}.
\end{proof}

Finally, we prove Theorem~\ref{unorirel}.
\begin{proof}[Proof of Theorem~\ref{unorirel}]
We just have to check that all the generators of $$\CC ob_{{}}(2, 4k-2) \cong {\mathfrak N}_{4k}^{2|\chi} \oplus \Z_2   \oplus \Z_2^{6k-3}$$
satisfy $\tau^1 + \tau^2 \equiv 0 \mod{2}$.
The generators represented by the fold maps 
$\tilde \va_{j,q, g} \co \tilde Q_{j,q, g}^{2+q} \to \R^2$, where $g$ runs over the elements of 
$\{ i_j, e_j^1, e_j^2 : 1 \leq j \leq q/2 \}$, satisfy $\tau^1 + \tau^2 \equiv 0 \mod{2}$
by Lemma~\ref{immgencong}.

By \cite{Br69} every class in ${\mathfrak N}_{4k}^{2|\chi}$ fibers over $S^2$.
Also, by Lemma~\ref{fiberedmaprelOKunori} any fold map $f_z \co Z \to \R^2$
where $Z$ fibers over $S^2$ satisfies $\tau^1(f_z) + \tau^2(f_z) \equiv 0 \mod{2}$.

We need one additional generator which gives us the possibility to generate the $\Z_2$ summand of 
${\mathfrak N}_{4k}^{2|\chi} \oplus \Z_2$ independently from any other summand.
The map $$f_C \co \CP^{2k} \# \CP^{2k} \to \R^2$$ provided by Proposition~\ref{prop2CP^2-nfoldmapproof-al} can serve this purpose
 because the classes
$$\Im_{2,4k-2}( [f_C]  ), \Im_{2,4k-2}( [e_j^1]),  \Im_{2,4k-2}( [e_j^2]),  \Im_{2,4k-2}( [i_j])$$
generate exactly the $\Z_2   \oplus \Z_2^{6k-3}$ part of $\CC ob(2, 4k-2)$ by Lemma~\ref{Z2partgen}.

Since this map $f_C \co \CP^{2k} \# \CP^{2k} \to \R^2$ also
satisfies $\tau^1(f) + \tau^2(f) \equiv 0 \mod{2}$,
we get the result.
\end{proof}

\end{document}